\documentclass{proc-l}

\usepackage{epsfig}
\usepackage{epstopdf}
\usepackage{graphicx}
\usepackage[leqno]{amsmath}
\usepackage{amssymb}
\usepackage{amsmath}
\usepackage{psfrag}
\usepackage{rotating}
\usepackage[usenames]{color}
\usepackage{multirow}
\usepackage{tikz}
\usetikzlibrary{cd}
\usetikzlibrary{calc}
\usepackage{verbatim}
\usepackage{multicol}

\usepackage{soul}

\newtheorem{theorem}{Theorem}[section]
\newtheorem{lemma}[theorem]{Lemma}

\theoremstyle{definition}
\newtheorem{definition}[theorem]{Definition}
\newtheorem{example}[theorem]{Example}

\theoremstyle{remark}
\newtheorem{remark}[theorem]{Remark}

\numberwithin{equation}{section}

\newtheorem{proposition}[theorem]{Proposition}
\newtheorem{corollary}[theorem]{Corollary}

\newcommand{\R}{\mathbb{R}}
\newcommand{\Z}{\mathbb{Z}}
\newcommand{\N}{\mathbb{N}}

\newcommand{\p}{\varphi}

\newcommand{\eps}{\varepsilon}
\newcommand{\dgm}{\mathrm{Dgm}}

\newcommand{\newatop}[2]{\genfrac{}{}{0pt}{1}{#1}{#2}}

\begin{document}

\title[The coherent matching distance in 2D persistent homology]{On the geometrical properties of the coherent matching distance in 2D persistent homology}

\author{Andrea Cerri}
\address{IMATI -- CNR, Genova, Italia}
\curraddr{}
\email{andrea.cerri@ge.imati.cnr.it}

\author{Marc Ethier}
\address{Facult\'e des Sciences, Universit\'e de Saint-Boniface, Winnipeg, Manitoba, Canada}
\curraddr{}
\email{methier@ustboniface.ca}

\author{Patrizio Frosini}
\address{Dipartimento di Matematica, Universit\`a di Bologna, Italia}
\curraddr{}
\email{patrizio.frosini@unibo.it}


\date{}

\dedicatory{}


\begin{abstract}
In this paper we study a new metric for comparing Betti numbers functions in bidimensional persistent homology, based on coherent matchings, i.e. families of matchings that vary in a continuous way. We prove some new results about this metric, including a property of stability. In particular, we show that the computation of this distance is strongly related to suitable filtering functions associated with lines of slope $1$, so underlining the key role of these lines in the study of bidimensional persistence. In order to prove these results, we introduce and study the concepts of extended Pareto grid for a normal filtering function as well as of transport of a matching. As a by-product, we obtain a theoretical framework for managing the phenomenon of monodromy in 2D persistent homology.
\end{abstract}

\maketitle


\section*{Introduction}
The classical approach to persistent homology is based on the study of the homological changes of the sublevel sets $X_{f\preceq w}$ of a topological space $X$ filtered by means of a continuous function $f:X\to \R^m$, when $w$ varies in $\R^m$. This theory is interesting both from the theoretical and applicative point of view, since the function $f$ can be used to describe both topological properties of $X$ and data defined on this space. A description of persistent homology and its use can be found in \cite{EdMo13}.

The case $m=2$ is intrinsically more difficult to study than the case $m=1$ and calls for the development of new mathematical ideas and methods. One of these methods consists in a reduction from the $2$-dimensional to the $1$-dimensional case by means of a family of functions $f_{(a,b)}:X\to\R$, with $a\in\ ]0,1[$ and $b\in\R$ (cf.~\cite{CeDi*13}), defined by setting $f_{(a,b)}(x):=\max\left\{\frac{f_1(x)-b}{a},\frac{f_2(x)+b}{1-a}\right\}$. Each pair $(a,b)$ identifies the positive slope line $r_{(a,b)}$  in $\R^2$ defined by the parametric equation $(u,v)=(at+b,(1-a)t-b). $
The function  $f_{(a,b)}$ allows one to represent the set $\{x\in X:f(x)\preceq (u,v)\}$ as the set $\{x\in X:f_{(a,b)}(x)\le t\}$, which describes a $1$-dimensional filtration of $X$ for $t$ varying in $\R$. For technical reasons, we normalize the function $f_{(a,b)}$ by setting $f^*_{(a,b)}(x):=\min\{a,1-a\}\cdot f_{(a,b)}(x)$.
In plain words, the previous 1D filtration associated with the function $f^*_{(a,b)}$ is obtained by
projecting $X$ to the plane $\R^2$ by means of $f$ and considering for each $p\in r_{(a,b)}$ the subset $X_p\subseteq X$ given by the points staying on the bottom left of $p$ (see Figure~\ref{1Dfilt}). It is well-known that in each degree $k$ the collection of the 1D Betti numbers functions associated with the 1D filtrations defined by the filtering functions $f^*_{(a,b)}$ is equivalent to the 2D Betti numbers function of $f$ \cite{CeDi*13}.

\begin{figure}[ht]
\begin{center}
\begin{tikzpicture}[scale=1.0]
\draw [->,line width=1.1] (0,-0.3) -- (0,5.3);
\draw [->,line width=1.1] (-0.3,0) -- (7.3,0);
\draw (7.0,-0.25) node {$f_1$};
\draw (-0.25,5.0) node {$f_2$};

\newcommand{\potato}{plot [smooth cycle] coordinates {(3.5,0.5) (5,1.0) (6,2) (5.5,3.5) (3.5,4.0) (2.5,5.0) (0.8,4.9) (0.4,3.5) (1.0,1.0)}}

\draw [black,fill=yellow] \potato;

\begin{scope}
	\clip \potato;
	\fill [color=blue!0.5!cyan] (0.2,0.2) rectangle (4.5,3.55);
\end{scope}
\draw \potato;

\draw [black,fill=white] (3.0,2.5) ellipse (0.8 and 0.4);

\draw (0.4,0.4) -- (6.0,4.7);
\draw (6.1,4.3) node {$r_{(a,b)}$};

\draw (4.5,3.55) node {$\bullet$};

\draw (0.2,3.55) -- (4.5,3.55) -- (4.5,0.2);
\draw (4.75,3.5) node {$p$};
\draw (3.5,1.0) node {$X_p$};
\draw (5.75,1.2) node {$X$};
\end{tikzpicture}
\end{center}
\caption{The 1D filtration $\{X_p\}_{p\in r_{(a,b)}}$ defined by the function $f^*_{(a,b)}$.
The light blue set $X_p$ is the sublevel set associated with the value $p$.}
\label{1Dfilt}
\end{figure}
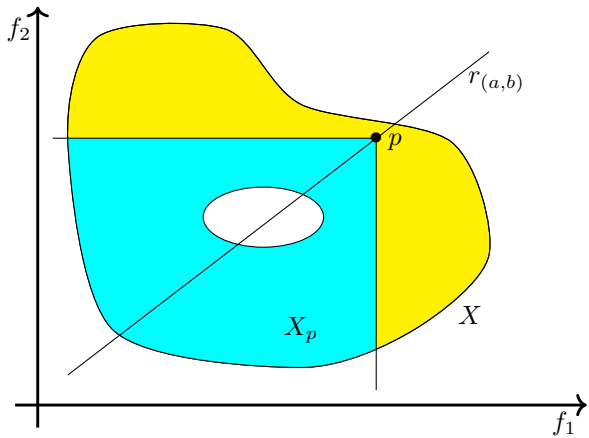

After fixing $k\in \N$, each $1$-dimensional filtration associated with the function $f^*_{(a,b)}$ defines a persistence diagram $\dgm\left(f_{(a,b)}^*\right)$, which is the set of pairs $(b_i,d_i)$ describing the time of birth $b_i$ and the time of death $d_i$ of the $i$th homological class in degree $k$ along the filtration associated with $f^*_{(a,b)}$. If two filtering functions $f,g:X\to\R^2$ are given, a common way to compare the two collections $\left\{\dgm\left(f_{(a,b)}^*\right)\right\}_{(a,b)\in ]0,1[\times\R}$ and $\left\{\dgm\left(g_{(a,b)}^*\right)\right\}_{(a,b)\in ]0,1[\times\R}$ consists in computing the supremum of the classical bottleneck distance between the persistence diagrams $\dgm\left(f_{(a,b)}^*\right)$ and $\dgm\left(g_{(a,b)}^*\right)$ over $(a,b)$. This idea leads to a metric $D_{\mathrm{match}}$ between the aforementioned families of persistence diagrams (cf. \cite{BiCe*11,CeDi*13}). We observe that, in principle, a small change of the pair $(a,b)$ can cause a large change in the ``optimal'' matching, that is, the matching realizing the bottleneck distance between $\dgm\left(f_{(a,b)}^*\right)$ and $\dgm\left(g_{(a,b)}^*\right)$. In other words, the definition of $D_{\mathrm{match}}$ is based on a family of optimal matchings that is not required to change continuously with respect to the pair $(a,b)$.

Experiments concerning the computation of this distance $D_{\mathrm{match}}$ reveal an interesting phenomenon, consisting of the fact that many examples exist where the supremum defining $D_{\mathrm{match}}(f,g)$ is taken for lines $r_{(a,b)}$ with $a\approx 1/2$. Figure~\ref{fattostrano} illustrates two of these examples.

\begin{figure}[ht]
\includegraphics[width=\textwidth]{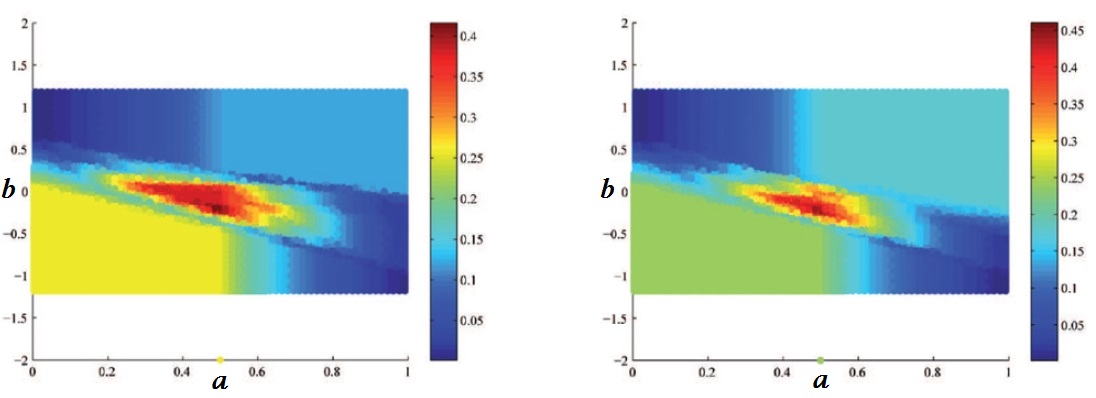}
\caption{The bottleneck distance between the persistence diagrams $\dgm\left(f_{(a,b)}^*\right)$ and $\dgm\left(g_{(a,b)}^*\right)$ for two different pairs $(f,g)$ of functions from $S^2$ to $\R^2$, represented as a function of $(a,b)$. The colors correspond to the value of the bottleneck distance at each point $(a,b)$, with red meaning higher values and blue, lower values. We can observe that the maximum value is taken at a point $\left(\bar a,\bar b\right)$ with $\bar a\approx 1/2$. More details about the considered functions can be found in \cite{BiCe*11}.}
\label{fattostrano}
\end{figure}

A natural question arises: Does the property illustrated in those examples always hold for the distance $D_{\mathrm{match}}$?

Unfortunately, we are not able to directly answer this question, because of the lack of geometrical properties in the definition of $D_{\mathrm{match}}$. Furthermore, we observe that
 while the metric $D_{\mathrm{match}}$ is rather simple to define and approximate by considering a suitable family of filtering functions associated with lines having positive slope, it has two main drawbacks. First, it forgets the natural link between the homological properties of filtrations associated with lines that are close to each other. As a consequence, part of the interesting homological information is lost. Second, its intrinsically discontinuous definition makes studying its properties difficult.

 For these reasons, in the previous paper \cite{CeEtFr16} we have introduced a new matching distance between 2D persistence diagrams (i.e. families of persistence diagrams associated with the lines $r_{(a,b)}$ as $(a,b)$ changes), called \emph{coherent matching distance} and based on matchings that change ``coherently'' with the filtrations we take into account. In other words, the basic idea consists of considering only matchings between the persistence diagrams $\dgm\left(f_{(a,b)}^*\right)$ and $\dgm\left(g_{(a,b)}^*\right)$ that change continuously with respect to the pair $(a,b)$. This requirement is both natural and useful, and this paper is devoted to the exploration of its main consequences.

First of all, the idea of ``coherent matching'' leads to the discovery of an interesting phenomenon of monodromy. We observe that when we require that the matchings change continuously, we have to avoid the pairs $(a,b)$ at which the persistence diagram contains double points, called \emph{singular pairs}. This is done by choosing a connected open set $U$ of regular (non-singular) pairs in the parameter space, and assuming that $(a,b)\in U$. In doing this, we can preserve the ``identity'' of points in the persistence diagram and follow them when we move in the parameter space. From this easily arises the concept of a family of matchings that is continuously changing. Interestingly, turning around a singular pair can produce a permutation in the considered persistence diagram, so that the considered filtering function is associated with a monodromy group.
A basic example of this monodromy phenomenon can be found by taking the filtering function $f=(f_1,f_2):X=\R^2\to \R^2$
with
$f_1(x,y)=x$, and
$$f_2(x,y) = \left\{
\begin{array}{ll}
-x & \mbox{ if } y=0\\
-x+1 & \mbox{ if } y=1\\
-2x & \mbox{ if } y=2\\
-2x+\frac{5}{4} & \mbox{ if } y=3\\
\end{array}\right.,$$
$f_2(x,y)$ then being extended linearly for every $x$ on the segments respectively joining $(x,0)$ with $(x,1)$, $(x,1)$ with $(x,2)$, and $(x,2)$ to $(x,3)$. On the half-lines $\{(x,y)\in \R^2: y<0\}$ and $\{(x,y)\in \R^2: y>3\}$, $f_2$ is then being taken with constant slope $-1$ in the variable $y$. The graph of $f_2$ is shown in Figure~\ref{figmon1}.

\begin{figure}
\begin{center}
\includegraphics[width=0.9\textwidth]{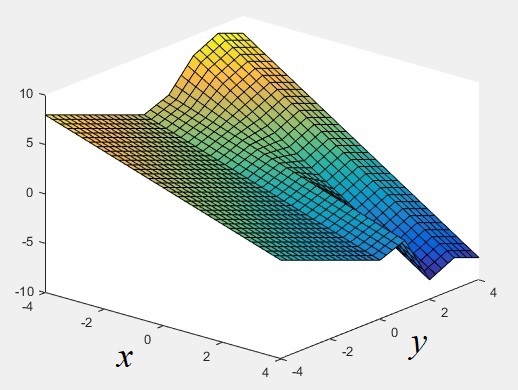}
\end{center}
\caption{The graph of $f_2$ in our basic example of monodromy for 2D persistent homology.}
\label{figmon1}
\end{figure}

The persistence diagram in degree $0$ of the function $f_{\left(1/4,0\right)}^*$ contains a double point, so that $\left(1/4,0\right)$ is a singular pair for $f$ in degree $0$. If we move around the point $(1/4,0)$ in the parameter space, we can see that two points of the persistence diagram $\dgm\left(f_{(a,b)}^*\right)$ exchange their position. For more details about this example we refer the interested reader to the paper~\cite{CeEtFr13}. We can easily adapt this example and get a smooth filtering function defined on a smooth closed manifold, revealing a similar phenomenon of monodromy.

As a consequence, our definition of ``coherent matching'' must take a monodromy group into account.
This is done in our paper by defining a transport operator $T^{(f,g)}_\pi$, which continuously transports each matching $\sigma_{(a,b)}$ between
the persistence diagrams $\dgm\left(f_{(a,b)}^*\right)$, $\dgm\left(g_{(a,b)}^*\right)$ to a matching $\sigma_{(a',b')}$
between
the persistence diagrams $\dgm\left(f_{(a',b')}^*\right)$, $\dgm\left(g_{(a',b')}^*\right)$ along a path $\pi$ from $(a,b)$ to $(a',b')$ in the set $U$. The existence of monodromy implies that the transport of $\sigma_{(a,b)}$ does not depend only on the pairs $(a,b)$, $(a',b')$ but also on the path $\pi$ we consider.

 By introducing the transport operator $T^{(f,g)}_\pi$, we can define the coherent cost
 $\mathrm{cohcost}\left(\sigma_{(a,b)}\right)$
 as the supremum of the classical cost of the matchings that we can obtain from $\sigma_{(a,b)}$ by means of every possible transport operator $T^{(f,g)}_\pi$ over $\pi$.

This done, the definition of the coherent matching distance $CD_U$ is straightforward: If two filtering functions $f,g:X\to\R^2$ are given and $U$ does not contain their singular pairs in the considered degree $k$, $CD_U(f,g)$  is the infimum of the coherent costs of the matchings between the sets $\dgm\left(f_{(a,b)}^*\right)$ and $\dgm\left(g_{(a,b)}^*\right)$ computed with respect to the chosen degree $k$, for a pair $(a,b)$ arbitrarily fixed. We also prove that this definition does not depend on the choice of $(a,b)$.

A key point in our paper consists in proving that the function $\mathrm{cost}\left(T^{(f,g)}_\pi\left(\sigma_{(a,b)}\right)\right)$ takes its global maximum over $\pi$ when the endpoint $\pi(1)$ of $\pi$ belongs to the vertical line $a=1/2$ or to the boundary of $U$ (Theorem~\ref{finalth} in Section~\ref{proofconjecture}). This result follows from the \emph{maximum principle for the coherent transport} (Theorem~\ref{maxprinc}) and casts new light on the abundance of examples where the supremum defining the classical distance $D_{\mathrm{match}}$ is taken for lines $r_{(a,b)}$ with $a\approx 1/2$.
In our opinion, the previous result can be seen as a strong signal that the coherent matching distance $CD_U$ should be preferred to the classical matching distance $D_{match}$ both in theory and applications,
since its use allows one to manage the parameter space $]0,1[\times\R$ more efficiently.
We observe that the value $\bar a=1/2$ identifies the planar lines with slope $1$. We think that the filtering functions associated with these lines are worth further study in 2D persistent homology, since they appear to encapsulate most relevant information. It is interesting to point out that these lines also take an important place in the paper \cite{CeLa16}, although in a different context, and that the direction of the vector $(1,1,\ldots,1)$ has a key role in the definition of interleaving distance between multidimensional persistence modules \cite{Le15}. The fact that lines of slope $1$ appear in various different approaches suggests to us that they would deserve further study. We observe that for $a=1/2$ the function $f^*_{(a,b)}$ coincides with the function $f_b:=\max\left\{f_1-b,f_2+b\right\}$, so our research suggests that this collection of filtering functions could play an important role in 2D persistent homology. Incidentally, this is also supported by the fact that, fixing $a=1/2$, it is possible to replace the classical upper bound for the distance between the 2-dimensional persistent Betti numbers, that is, $\|f-g\|_{\infty}$, by $\max_b\|f_b-g_b\|_{\infty}$ (Proposition~\ref{statement}).

We conclude by observing that, while our research highlights the importance of the lines of slope $1$, this does not mean that lines with a different slope are useless in 2D persistent homology. As we will show,
the construction of matchings that change coherently with filtrations defined by lines of slope $1$ compels us to use lines with slope different from $1$ as well. This is due to the need to avoid lines possibly corresponding to singular pairs. Furthermore, the phenomenon of monodromy can appear only if lines with a slope different from $1$ are also considered. These facts justify our approach, which is based on the use of every line of positive slope.

 Our paper is devoted to illustrating the theoretical model that we have sketched in this introduction. This will require the use of several new concepts and the proof of many properties related to these concepts, so that a by-product of our research is the development of a new theoretical framework to manage 2D persistent homology, based on the concept of extended Pareto grid.

 The outline of the paper is as follows. In Section~\ref{MS} we recall the necessary mathematical background. In Section~\ref{2DS} we illustrate the 2D setting for persistent Betti numbers functions. In Section~\ref{preparing} we introduce the concept of extended Pareto grid as the main mathematical tool in our approach, and prove several results paving the way to the mathematical framework illustrated in the following sections.
 In Section~\ref{CMD} we introduce the concept of transport of a matching together with its main properties, and present the definition of the coherent 2-dimensional matching distance, also proving its stability.
 In Section~\ref{proofconjecture} we prove the maximum principle for the coherent transport and present our main result on the coherent matching distance in 2D persistent homology (Theorem~\ref{finalth}).
 In Section~\ref{relCUDmatch} we conclude the paper by illustrating the relation between the coherent matching distance and the classical matching distance.
 \smallskip


\subsection*{Related literature}
Studying the persistence properties of vector-valued functions is usually referred to as \emph{multidimensional persistence}. These concepts were first investigated in \cite{FrMu99} with respect to homotopy groups; multidimensional persistence modules were then considered in \cite{CaZo09}, and subsequently studied in other papers including \cite{CaSiZo09} and the recent \cite{Le15,LeWr15}. Another approach to the multidimensional setting is the one proposed in \cite{BiCe*08}. Focusing on 0th homology, the authors introduce a procedure allowing for a reduction of the multidimensional case to the 1-dimensional setting by using a suitable family of derived real-valued filtering functions. Such a result has been partially extended in \cite{CaDiFe10}, i.e. for any homology degree but restricted to the case of max-tame filtering functions, and then further refined in \cite{CeDi*13} for continuous filtering functions. This approach leads to the definition of a multidimensional matching distance between persistent Betti numbers functions and to algorithms for its computation (cf. \cite{BiCe*11,CeFr11}). More recently, the interleaving distance between multidimensional persistence modules has been formally introduced and discussed in \cite{Le15}. However, according to the author of \cite{Le15}, the question
of if and how this last distance can be computed or approximated remains open, thus justifying the study of other metrics such as the one we propose in this paper. In the same line of thought, some recent papers have been devoted to the computation of bounds for the interleaving distance \cite{Bj16,BoLe16,DeXi18}. The phenomenon of monodromy in 2D persistent homology has been described and studied in \cite{CeEtFr13}.

\section{Mathematical setting}\label{MS}

In what follows we will assume that $f=(f_1,f_2)$ is a continuous map from a finitely triangulable topological space $M$ to the real plane $\R^2$.

\subsection{Persistent Betti numbers}\label{PBN}

As a reference for multidimensional persistent Betti numbers we use \cite{CeDi*13}. According to the main topic of this paper, we will also stick to the notations and working assumptions adopted in \cite{CeEtFr13}. In particular, we build on the strategy adopted in the latter paper to study certain instances of monodromy for multidimensional persistent Betti numbers. Roughly, the idea is to reduce the problem to the analysis of a collection of persistent Betti numbers associated with real-valued functions, and to their compact representation in terms of \emph{persistence diagrams}.

We use the following notations: $\Delta^+$ is the open set $\{(u,v)\in\R\times\R:u< v\}$. $\Delta$ represents the diagonal $\{(u,v)\in\R\times\R:u= v\}$. We can extend $\Delta^+$ with points at infinity of the kind $(u,\infty)$, where $|u|<\infty$. Denote this set $\Delta^*$. For a continuous function $\p:M\to\R$, and for any $k\in\mathbb{N}$, if $u<v$, the inclusion map of the sublevel set $M_u=\{x\in M:\p(x)\leq u\}$ into the sublevel set $M_v=\{x\in M:\varphi(x)\leq v\}$ induces a homomorphism from the $k$th homology group of $M_u$ into the $k$th homology group of $M_v$. The image of this homomorphism is called the {\em $k$th persistent homology group of $(M,\p)$ at $(u,v)$}, and is denoted by $H_k^{(u,v)}(M,\p)$. In other words, the group $H_k^{(u,v)}(M,\p)$ contains all and only the homology classes of $k$-cycles born before or at $u$ and still alive at $v$. By assuming that coefficients are chosen in a field $\mathbb{K}$, we get that homology groups are vector spaces. Therefore, they can be completely described by their dimension, leading to the following definition \cite{EdLeZo02}.

\begin{definition}[Persistent Betti Numbers]\label{Rank}
The {\em persistent Betti numbers function} of $\p$ in degree $k$, briefly PBN, is the function $\beta_{\p}:\Delta^+\to\mathbb{N}\cup\{\infty\}$ defined by
\begin{displaymath}
\beta_{\p}(u,v)=\dim H_k^{(u,v)}(M,\p).
\end{displaymath}
\end{definition}
Under the above requirements for $M$, it is possible to show that $\beta_{\p}$ is finite for all $(u,v)\in\Delta^+$ \cite{CeDi*13}. Obviously, for each $k\in\mathbb{N}$, we have different PBNs of $\p$ (which might be denoted by $\beta_{\p,k}$, say), but for the sake of notational simplicity we omit adding any reference to $k$.

Following \cite{CeDi*13}, we assume the use of \v{C}ech homology, and refer the reader to that paper for a detailed explanation about preferring this homology theory to others. For the present work, it is sufficient to recall that, with the use of \v{C}ech homology, the PBNs of a real-valued
function can be completely described by the corresponding \emph{persistence diagrams}. Formally, a persistence diagram can be defined via the notion of \emph{multiplicity} \cite{CoEdHa07,FrLa01}. Following the convention used for PBNs, any reference to $k$ will be dropped in the sequel.

\begin{definition}[Multiplicity]\label{Multiplicity}
The \emph{multiplicity} $\mu_{\p}(u,v)$  of $(u,v)\in\Delta^+$ is the finite, non-negative number given by
\begin{equation*}
\min_{\newatop{\eps>0}{u+\eps<v-\eps}} \beta_{\p}(u+\eps ,v-\eps)-\beta_{\p}(u-\eps ,v-\eps)-\beta_{\p}(u+\eps,v+\eps)+\beta_{\p}(u-\eps ,v+\eps).
\end{equation*}
The \emph{multiplicity} $\mu_{\varphi}(u,\infty)$  of $(u,\infty)$ is the finite, non-negative number given by {\setlength\arraycolsep{1pt}
\begin{equation*}
\min_{\eps > 0,\,u+\eps<v} \beta_{\p}(u+\eps,v)-\beta_{\p}(u-\eps ,v).
\end{equation*}}
\end{definition}
\begin{definition}[Persistence Diagram]\label{persDiag}
The persistence diagram $\dgm(\p)$ is the multiset of all points $(u,v)\in\Delta^*$ such that $\mu_{\p}(u,v)>0$, counted with their multiplicity, union the singleton $\{\Delta\}$, where the point $\Delta$ is counted with infinite multiplicity.
\end{definition}
Each point $(u,v)\in \dgm(\p)\cap\Delta^+$ will be called \emph{proper}, while each point $(u,\infty)\in \dgm(\p)$ will be called \emph{a point at infinity} or \emph{an improper point}.

\begin{remark}\label{remDelta}
In literature, persistence diagrams are usually defined to contain each single point of the diagonal $\Delta$ instead of one point representing the whole diagonal, with infinite multiplicity. The two definitions are equivalent, but we prefer the latter because it will allow us to make easier our exposition and in particular the definition of the set $\mathcal{F}_{U,c}$ in Section~\ref{preparing}.
\end{remark}

We endow $\Delta^*\cup \{\Delta\}$ with the following extended metric $d$.
We define
\begin{equation}\label{deltaDistance}
d\left(\left(u,v\right),\left(u',v'\right)\right):=\min\left\{\max\left\{|u-u'|,|v-v'|\right\},\max\left\{\frac{v-u}{2},\frac{v'-u'}{2}\right\}\right\}
\end{equation}
for every
$\left(u,v\right),\left(u',v'\right)\in\Delta^*$, with the convention about points at infinity that $\infty - v =  v-\infty = \infty$ when $v\neq \infty$, $\infty -\infty =0$, $\frac{\infty}{2} = \infty$, $|\infty| = \infty$, $\min\{c,\infty\} = c$ and $\max\{\infty,c\}=\infty$.  Furthermore, we set $d((u,v),\Delta):= \infty$ if $v=\infty$, $d((u,v),\Delta):= \frac{v-u}{2}$ if $v<\infty$, and $d(\Delta,\Delta):=0$.

Persistence diagrams are stable under the \emph{bottleneck distance} (a.k.a. \emph{matching distance}). Roughly, small changes in the considered function $\p$ induce small changes in the position of the points of $\dgm(\p)$ which are far from the diagonal, and possibly
produce variations close to the diagonal \cite{CoEdHa07,dAFrLa10}. A visual  intuition of this fact is given in Figure~\ref{bottDist}. Formally, we have the following definition:
\begin{definition}[Bottleneck distance]\label{matchDist}
Let $\dgm (\varphi)$, $\dgm (\psi)$ be two persistence diagrams.
For each bijection $\sigma$ between $\dgm(\varphi)$ and
$\dgm (\psi)$ we set $\mathrm{cost}(\sigma):=\max_{X\in
\dgm (\varphi)}d(X,\sigma(X))$.
The bottleneck
distance $d_B\left(\dgm (\varphi),\dgm (\psi)\right)$ is defined as
\begin{equation*}
d_B(\dgm(\varphi),\dgm(\psi))=\min_{\sigma}\mathrm{cost}(\sigma),
\end{equation*}
where $\sigma$ varies among all the bijections between $\dgm(\varphi)$ and
$\dgm (\psi)$.
\end{definition}
In practice, the distance $d$ defined in (\ref{deltaDistance}) compares the cost of moving a point $X$ to a point $Y$ with that of annihilating them by moving both $X$ and $Y$ onto $\Delta$, and takes the most convenient. Therefore, $d(X,Y)$ can be considered a measure of the minimum cost of moving $X$ to $Y$ along two different paths.

Sometimes in literature the definition of $\mathrm{cost}(\sigma)$ is given by means of a supremum instead of a maximum, and
the bottleneck distance $d_B(\dgm(\varphi),\dgm(\psi))$ is introduced as an infimum instead of a minimum. We underline that both these presentations are correct, as pointed out in \cite{CeDi*13}. In other words, a matching $\bar\sigma:\dgm(\varphi)\to\dgm(\psi)$ and a point $\bar X\in \dgm(\varphi)$ always exist, such that $d_B(\dgm(\varphi),\dgm(\psi))=\mathrm{cost}(\bar\sigma)=d(\bar X,\bar \sigma(\bar X))$. The matching $\bar \sigma$ is called an \emph{optimal matching} between $\dgm(\varphi)$ and $\dgm(\psi)$.

The stability of persistence diagrams can then be formalized as follows~\cite{CoEdHa07,dAFrLa10}:
\begin{theorem}[Stability Theorem]\label{stability}
Let $\varphi,\psi:M\to\R$ be two continuous functions. Then
$d_B\left(\dgm(\varphi),\dgm(\psi)\right)\leq\|\varphi-\psi\|_{\infty}
$.
\end{theorem}

\begin{figure}[ht]
\begin{center}
\begin{tikzpicture}[scale=0.85]
\draw [line width=1.1] (0,0) -- (8,0);
\draw (-0.2,0) node {$M$};
\draw [->,line width=1.1] (8.5,0) -- (13.8,0);
\draw [->,line width=1.1] (8.5,0) -- (8.5,5.3);
\draw (13.8,-0.25) node {$u$};
\draw (8.25,5.3) node {$v$};

\draw [cyan] plot [smooth,tension=1] coordinates {(0,1.3) (2.0,4.8) (4.5,0.5) (6.5,3.5) (8.0,0.8)};

\draw [dotted] (0,1.3) -- (9.8,1.3);
\draw [dotted] (1.9,4.81) -- (9.8,4.81) -- (9.8,1.3);
\draw [dotted] (4.7,0.47) -- (8.97,0.47);
\draw [dotted] (8.0,0.8) -- (9.3,0.8);
\draw [dotted] (6.5,3.5) -- (9.3,3.5) -- (9.3,0.8);

\draw [cyan,line width=0.8] (8.97,0.46) -- (8.97,5.3);
\draw [line width=0.1,black,fill=cyan] (9.8,4.81) circle (0.04);
\draw [line width=0.1,black,fill=cyan] (9.3,3.5) circle (0.04);

\draw [magenta] plot [smooth,tension=0.3] coordinates {(0,1.1) (0.7,3.5) (1.7,4.6) (2.2,4.4) (2.5,5.0) (4.1,0.2) (4.6,0.7) (4.9,0.5) (6.2,3.5) (6.6,3.1) (7.1,3.3) (7.5,2.4) (7.9,0.9)};

\draw [dotted] (0,1.1) -- (9.6,1.1);
\draw [dotted] (2.18,4.4) -- (12.9,4.4);
\draw [dotted] (1.73,4.61) -- (12.9,4.61) -- (12.9,4.4);
\draw [dotted] (2.43,5.08) -- (9.6,5.08) -- (9.6,1.1);
\draw [dotted] (4.2,0.11) -- (8.61,0.11);
\draw [dotted] (4.6,0.7) -- (8.93,0.7) -- (8.93,0.43);
\draw [dotted] (4.81,0.43) -- (8.93,0.43);
\draw [dotted] (7.9,0.9) -- (9.4,0.9);
\draw [dotted] (6.26,3.55) -- (9.4,3.55) -- (9.4,0.9);
\draw [dotted] (6.6,3.1) -- (11.6,3.1);
\draw [dotted] (7.08,3.31) -- (11.6,3.31) -- (11.6,3.1);

\draw [magenta,line width=0.8] (8.61,0.12) -- (8.61,5.3);
\draw [line width=0.1,black,fill=magenta] (12.9,4.61) circle (0.04);
\draw [line width=0.1,black,fill=magenta] (9.6,5.08) circle (0.04);
\draw [line width=0.1,black,fill=magenta] (8.93,0.7) circle (0.04);
\draw [line width=0.1,black,fill=magenta] (9.4,3.55) circle (0.04);
\draw [line width=0.1,black,fill=magenta] (11.6,3.31) circle (0.04);

\draw [line width=1.1] (8.5,0) -- (13.8,5.3);

\draw[color = cyan] (1.0,4.6) node {$\varphi$};
\draw[color = magenta] (3.1,4.1) node {$\psi$};
\end{tikzpicture}
\end{center}
\caption{Changing the function $\varphi$ to $\psi$ induces a change in the persistence diagram. In this example, the graphs on the left represent the real-valued functions $\varphi$ and $\psi$, defined on a space $M$ (a segment). The corresponding persistence diagrams (restricted to 0th homology) are displayed on the right.}\label{bottDist}
\end{figure}
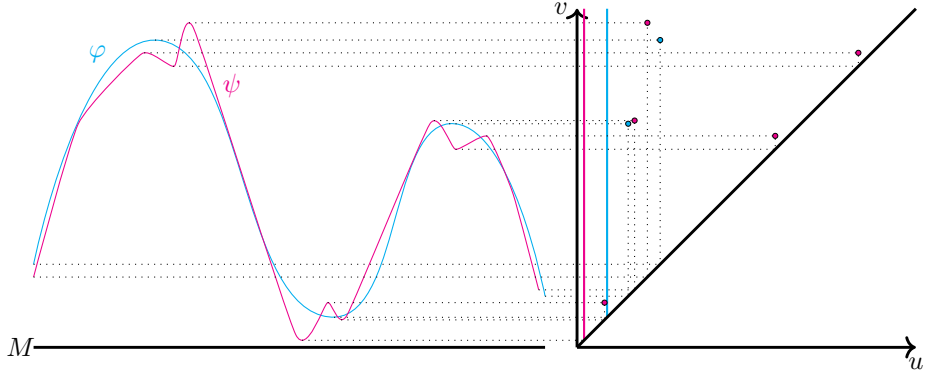

\section{2-dimensional setting}\label{2DS}
The definition of persistent Betti numbers can be easily extended to functions taking values in $\R^2$ \cite{CeDi*13}.
For a continuous function $f=(f_1,f_2):M\to\R^2$, and for any $k\in\mathbb{N}$, if $u_1<v_1$ and $u_2<v_2$, the inclusion map of the sublevel set $M_{(u_1,u_2)}:=\{x\in M:f_1(x)\leq u_1, f_2(x)\leq u_2\}$ into the sublevel set $M_{(v_1,v_2)}:=\{x\in M:f_1(x)\leq v_1, f_2(x)\leq v_2\}$ induces a homomorphism from the $k$th homology group of $M_{(u_1,u_2)}$ into the $k$th homology group of $M_{(v_1,v_2)}$. The image of this homomorphism is called the {\em $k$th persistent homology group of $(M,f)$ at $((u_1,u_2),(v_1,v_2))$}, and is denoted by $H_k^{((u_1,u_2),(v_1,v_2))}(M,f)$.

\begin{definition}[Persistent Betti Numbers in the case $m=2$]\label{Rank2}
The {\em persistent Betti numbers function} of $f=(f_1,f_2):M\to\R^2$ in degree $k$, briefly PBN, is the function $\beta_{f}:\{((u_1,u_2),(v_1,v_2))\in \R^2\times\R^2:u_1<v_1,u_2<v_2\}\to\mathbb{N}\cup\{\infty\}$ defined by
\begin{displaymath}
\beta_{f}((u_1,u_2),(v_1,v_2))=\dim H_k^{((u_1,u_2),(v_1,v_2))}(M,f).
\end{displaymath}
\end{definition}
We discuss this for the specific case of the above function $f:M\to\R^2$, referring the reader to Figure~\ref{foliation} for a pictorial representation.

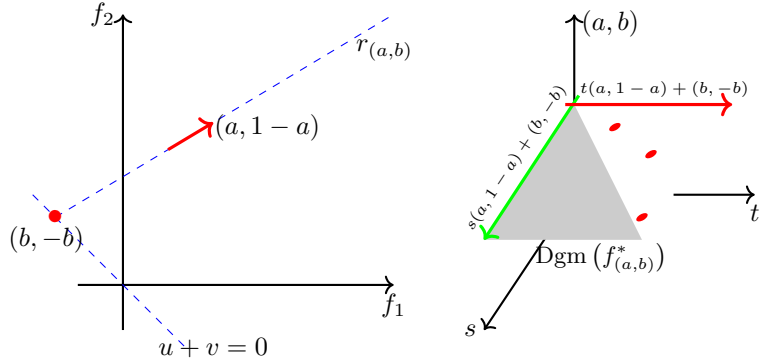
\begin{figure}[ht]
\begin{center}

\begin{tikzpicture}[scale=0.6]
\draw[->,line width=0.8] (-1,0) -- (6,0);
\draw[->,line width=0.8] (0,-1) -- (0,6);
\draw (6.0,-0.45) node {$f_1$};
\draw (-0.45,6.0) node {$f_2$};
\draw[color=blue,dashed] (-2,2) -- (1.4,-1.4);
\draw[color=blue,dashed] (-1.5,1.5) -- (6,6);
\draw[color=red] (-1.5,1.5) node {{\Large $\bullet$}};
\draw (-1.7,1.0) node {$(b,-b)$};
\draw (2.0,-1.4) node {$u+v=0$};
\draw[->, color=red,line width=1.2] (1.0,3.0) -- (2.0,3.6);
\draw (3.2,3.5) node {$(a,1-a)$};
\draw (5.8,5.2) node {$r_{(a,b)}$};

\draw[->,line width=0.8] (10,4) -- (10,6);
\draw[->,line width=0.8] (12.2,2) -- (14,2);
\draw[->,line width=0.8] (10,2) -- (8,-1);
\fill[gray!40] (10,4) -- (8,1) -- (11.5,1) -- cycle;
\fill [color=red,rotate around={30:(10.9,3.5)}] (10.9,3.5) circle (0.14 and 0.07);
\fill [color=red,rotate around={30:(11.7,2.9)}] (11.7,2.9) circle (0.14 and 0.07);
\fill [color=red,rotate around={30:(11.5,1.5)}] (11.5,1.5) circle (0.14 and 0.07);
\draw (10.6,0.6) node {{\small $\dgm\left(f_{(a,b)}^*\right)$}};
\draw[->,color=green,line width=1.2] (10.1,4.2) -- (8,1) node[midway,above,sloped,color=black] {{\tiny $s(a,1-a)+(b,-b)$}};
\draw[->,color=red,line width=1.2] (9.8,4.0) -- (13.5,4.0);
\draw (10.8,5.8) node {$(a,b)$};
\draw (14.0,1.6) node {$t$};
\draw (7.7,-1) node {$s$};
\draw (12.0,4.35) node {{\tiny $t(a,1-a)+(b,-b)$}};
\end{tikzpicture}
\end{center}
\caption{Correspondence between an admissible line $r_{(a,b)}$
and the persistence diagram $\dgm\left(f_{(a,b)}^*\right)$. Left: a 1D filtration is constructed by sweeping the line $r_{(a,b)}$. The vector $(a,1-a)$ and the point $(b,-b)$ are used to parameterize this line as $r_{(a,b)}: t\cdot (a,1-a) + (b,-b)$. Right: the persistence diagram of the 1D filtration can be found on a planar section of the domain of the 2D persistent Betti numbers function $\beta_f$.}
\label{foliation}
\end{figure}

Let us consider the set $\Lambda^+$ of all lines of $\R^2$ that have positive slope. This set can be parameterized by the set ${\mathcal {P}}(\Lambda^+)=\,]0,1[\,\times\R$, by taking each line $r\in\Lambda^+$ to the unique pair $(a,b)$ with $0<a<1$ and $b\in \R$ such that $(a,1-a)$ is a direction vector of $r$ and $(b,-b)\in r$.
The line $r$ will be denoted by $r_{(a,b)}$.
$\Lambda^+$ is referred to as the \emph{set of admissible lines}.
Each point $(u,v)=(u(t),v(t))=t\cdot (a,1-a)+(b,-b)$ of $r_{(a,b)}$ can be associated with the subset $M^{a,b}_t:=M_{(u(t),v(t))}$, that is the set
of the points of $M$ ``whose image by $f$ is under and on the left of $(u(t),v(t))$'' while $(u(t),v(t))$ moves along the line $r_{(a,b)}$. As a consequence, each admissible line $r_{(a,b)}$ defines a filtration $\{M^{a,b}_t\}$ of $M$ and a persistence diagram associated with this filtration.
The family of the persistence diagrams associated with the lines $r_{(a,b)}$ is called the \emph{2D persistence diagram of $f$}.

It is interesting to observe that the filtration $\{M^{a,b}_t\}$ can be also defined as the sublevel sets filtration induced by a suitable real-valued function.
In fact, we have that $M^{a,b}_t=\{x\in M: f_{(a,b)}(x)\leq t\}$ where $f_{(a,b)}:M\to\R$ is defined by setting $f_{(a,b)}(x):=\max\left\{\frac{f_1(x)-b}{a},\frac{f_2(x)+b}{1-a}\right\}$. The Reduction Theorem proved in \cite{CeDi*13} states that the persistent Betti numbers function $\beta_{f}$ can be completely recovered by considering all and only the persistent Betti numbers functions $\beta_{f_{(a,b)}}$ associated with the admissible lines $r_{(a,b)}$, which are in turn encoded in the corresponding persistence diagrams $\dgm\left(f_{(a,b)}\right)$.

In some sense, the study of persistent homology for $\R^2$-valued functions can be seen as the study of the persistent homology groups associated with the filtrations defined by the lines $r_{(a,b)}$, varying $(a,b)$ in ${\mathcal {P}}(\Lambda^+)$. It is natural to wonder which pairs $(a,b)$ are more relevant for the topological comparison of two functions $f,g:M\to \R^2$.
This paper is mainly devoted to underline the particular importance of the pairs $(a,b)\in {\mathcal {P}}(\Lambda^+)$ with $a=1/2$, starting from the following results providing an alternative, yet equivalent, formulation of the  $L^{\infty}$-distance between $f$ and $g$:

\begin{lemma}\label{lemmafab}
For every $(a,b)\in {\mathcal {P}}(\Lambda^+)$ set $f_{(a,b)}^*:=\min\{a,1-a\}\cdot f_{(a,b)}$.
Then $\left\|f_{(a,b)}^*- g_{(a,b)}^*\right\|_\infty\le \|f-g\|_\infty$.
\end{lemma}
\begin{proof}\label{prooflemmafab}
For every $(a,b)\in {\mathcal {P}}(\Lambda^+)$ and every $x\in M$, we have
$$
\begin{array}{lll}
\left|f_{(a,b)}^*(x)-g_{(a,b)}^*(x)\right| & = & \min\{a,1-a\}\cdot\left|f_{(a,b)}(x)-g_{(a,b)}(x)\right|  \\
&\le & \min\{a,1-a\}\cdot\max\left\{\left|\frac{f_1(x)-{g}_1(x)}{a}\right|, \left|\frac{f_2(x)-{g}_2(x)}{1-a}\right|\right\}  \\
&\le & \max\left\{\left|f_1(x)-{g}_1(x)\right|, \left|f_2(x)-{g}_2(x)\right|\right\}.
\end{array}
$$
\end{proof}

\begin{proposition}\label{statement}
Let $f,g:M\to\R^2$ be two continuous functions.
Then
$$
\|f-g\|_\infty=
\sup_{a,b}\left\|f^*_{(a,b)}-g^*_{(a,b)}\right\|_\infty=
\sup_{b}\left\|f^*_{(1/2,b)}-g^*_{(1/2,b)}\right\|_\infty.
$$
\end{proposition}

\begin{proof}
By Lemma~\ref{lemmafab}, we know that if $(a,b)\in {\mathcal {P}}(\Lambda^+)$ then
$\left\|f_{(a,b)}^*- g_{(a,b)}^*\right\|_\infty\le \|f-g\|_\infty$. Therefore
we have that
$$
\|f-g\|_\infty\ge
\sup_{a,b}\left\|f^*_{(a,b)}-g^*_{(a,b)}\right\|_\infty\ge
\sup_{b}\left\|f^*_{(1/2,b)}-g^*_{(1/2,b)}\right\|_\infty.
$$
Let us take a point $\bar x\in M$ such that $\|f-g\|_\infty=\|f(\bar x)-g(\bar x)\|_\infty$.
We can assume that $\|f(\bar x)-g(\bar x)\|_\infty=|f_1(\bar x)-g_1(\bar x)|$.
If $a=1/2$ then $\min\{a,1-a\}=a=1-a$, so that
$f^*_{(a,b)}(\bar x)=\max\{f_1(\bar x)-b,f_2(\bar x)+b\}$ and
$g^*_{(a,b)}(\bar x)=\max\{g_1(\bar x)-b,g_2(\bar x)+b\}$.

Furthermore, if we also assume that $b<\min f_1,-\max f_2,\min g_1,-\max g_2$ then
$f^*_{(a,b)}(\bar x)=f_1(\bar x)-b$ and $g^*_{(a,b)}(\bar x)=g_1(\bar x)-b$.
It follows that
$$
\sup_{b}\left\|f^*_{(1/2,b)}-g^*_{(1/2,b)}\right\|_\infty\ge
|f_1(\bar x)-g_1(\bar x)|=\|f-g\|_\infty.
$$
\end{proof}

\subsubsection{2-dimensional matching distance}\label{2DMD}

Assume now that we have two continuous functions $f,g:M\to\R^2$. We consider the persistence diagrams $\dgm\left(f_{(a,b)}\right)$, $\dgm\left(g_{(a,b)}\right)$ associated with the admissible line $r_{(a,b)}$, and normalize them by multiplying their points by $\min\{a,1-a\}$. This is equivalent to consider the \emph{normalized persistence} diagrams $\dgm\left(f_{(a,b)}^*\right)$, $\dgm\left(g_{(a,b)}^*\right)$, with $f^*_{(a,b)}=\min\{a,1-a\}\cdot f_{(a,b)}$ and $g^*_{(a,b)}=\min\{a,1-a\}\cdot g_{(a,b)}$, respectively. The 2-dimensional matching distance $D_{\mathrm{match}}(f,g)$ \cite{BiCe*11} is then defined as
$$
D_{\mathrm{match}}(f,g)=\sup_{(a,b)\in{\mathcal {P}}(\Lambda^+)}d_B\left(\dgm\left(f_{(a,b)}^*\right),\dgm\left(g_{(a,b)}^*\right)\right),
$$
with $d_B\left(\dgm\left(f_{(a,b)}^*\right),\dgm\left(g_{(a,b)}^*\right)\right)$ denoting the bottleneck distance between the normalized persistence diagrams $\dgm\left(f_{(a,b)}^*\right)$ and $\dgm\left(g_{(a,b)}^*\right)$.

\begin{remark}
It is common in the literature (cf. \cite{CeDi*13}) to refer to the 2-dimensional matching distance $D_{\mathrm{match}}$ as giving a distance between two 2-dimensional persistent Betti numbers functions (or 2D persistence diagrams). In this paper, in order to simplify the exposition, it will be said to give a pseudo-distance between the functions $f$ and $g$ themselves, denoted $D_{\mathrm{match}}(f,g)$. The same will be the case for the coherent matching distance $CD_U$ which will be defined in Section~\ref{secdefCMD}.
\end{remark}

By Lemma~\ref{lemmafab} and the Stability Theorem~\ref{stability} the next result immediately follows.

\begin{corollary}\label{corlemmafab}
$D_{\mathrm{match}}(f,g)\le \|f-g\|_\infty$.
\end{corollary}

\begin{remark}\label{remNormStability}
The introduction of normalized persistence diagrams in the definition of $D_{\mathrm{match}}$ is crucial to obtain a stable pseudo-metric
(cf. \cite[Thm.~4.4]{CeDi*13}). Indeed, Lemma~\ref{lemmafab} implies that the bottleneck distance $d_B\left(\dgm\left(f_{(a,b)}^*\right),\dgm\left(g_{(a,b)}^*\right)\right)$ is less than or equal to $\|f-g\|_\infty$, while we underline that this is not true for the distance $d_B\left(\dgm\left(f_{(a,b)}\right),\dgm\left(g_{(a,b)}\right)\right)$.
\end{remark}

\subsubsection{Monodromy in 2-dimensional persistent homology}
Since each function $f_{(a,b)}^*$ depends continuously on the parameters $a$ and $b$ with respect to the $\sup$-norm, it follows that the set of points in $\dgm\left(f_{(a,b)}^*\right)$ depends continuously on the parameters $a$ and $b$. Analogously, the set of points in $\dgm\left(g_{(a,b)}^*\right)$ depends continuously on the parameters $a$ and $b$. Suppose that $\sigma_{(a,b)}:\dgm\left(f_{(a,b)}^*\right)\to \dgm\left(g_{(a,b)}^*\right)$ is an \emph{optimal matching}, i.e. one of the matchings achieving the bottleneck distance $d_B\left(\dgm\left(f_{(a,b)}^*\right),\dgm\left(g_{(a,b)}^*\right)\right)$. Given the above arguments, a natural question arises, whether $\sigma_{(a,b)}$ changes continuously under variations of $a$ and $b$. In other words, we wonder if it is possible to straightforwardly introduce a notion of \emph{coherence} for optimal matchings with respect to the elements of ${\mathcal {P}}(\Lambda^+)$.

Perhaps surprisingly, the answer is no. A first obstruction is given by the fact that, trying to continuously extend a matching $\sigma_{(a,b)}$, the identity of points in the (normalized) persistence diagrams is not preserved when considering an admissible pair $\left(\bar a,\bar b\right)$ for which either $\dgm\left(f_{\left(\bar a,\bar b\right)}^*\right)$ or $\dgm\left(g_{\left(\bar a,\bar b\right)}^*\right)$ has points with multiplicity greater than 1. In other words, we cannot follow the path of a point of a persistence diagram when it collides with another point of the same persistence diagram. On the one hand, this problem can be solved by fixing a degree $k$ and replacing ${\mathcal {P}}(\Lambda^+)$ with its subset $\mathrm{Reg}(f)\cap \mathrm{Reg}(g)$, where $\mathrm{Reg}(f)$ is the set of all points $(a,b)\in {\mathcal {P}}(\Lambda^+)$ such that in degree $k$ the persistence diagram $\dgm\left(f_{(a,b)}^*\right)\setminus \{\Delta\}$ does not contain multiple points.
Throughout the rest of the paper, we will talk about \emph{singular pairs for $f$ in degree $k$} to denote the pairs $(a,b)\in \mathrm{Sing}(f):={\mathcal {P}}(\Lambda^+)\setminus \mathrm{Reg}(f)$, and about \emph{regular pairs for $f$ in degree $k$} to denote the pairs $(a,b)\in \mathrm{Reg}(f)$. An analogous convention holds referring to the singular and regular pairs for $g$.

On the other hand, however, continuously extending a matching $\sigma_{(a,b)}$ presents some problems even in this setting. Roughly, the process of extending $\sigma_{(a,b)}$ along a path $\pi:[0,1]\to\mathrm{Reg}(f)\cap\mathrm{Reg}(g)$ depends on the homotopy class of $\pi$ relative to its endpoints. This phenomenon is referred to as \emph{monodromy in 2-dimensional persistent homology}, and has been studied for the first time in \cite{CeEtFr13}. In the following we will show how to overcome this issue in order to define a \emph{coherent} modification of the standard 2-dimensional matching distance $D_{\mathrm{match}}$.

There are two different ways we can alleviate the difficulty caused by the monodromy phenomenon in order to construct a coherent 2-dimensional matching distance. We can choose to transport matchings by moving along paths in a covering of the parameter space, or we can rather define the transport of matchings along paths in the parameter space itself. In this paper we will choose this last approach.

\section{The extended Pareto grid and its main properties}\label{preparing}

In order to proceed we will assume that $M$ is a closed smooth manifold and our filtering function $f:M\to\R^2$ is sufficiently regular, in the sense described in this section. If not differently stated, we will also assume that a degree $k$ has been fixed for the computation of persistence diagrams.
\medskip

Let $f=(f_1,f_2)$ be a smooth map from a closed $C^{\infty}$-manifold $M$ of dimension $r\geq 2$ to the real plane $\R^2$. Choose a Riemannian metric on $M$ so that we can define gradients for $f_1$ and $f_2$. The \emph{Jacobi set} $\mathbb{J}(f)$ is the set of all points $p\in M$ at which the gradients of $f_1$ and $f_2$ are linearly dependent, namely $\nabla f_1(p)=\lambda\nabla f_2(p)$ or $\nabla f_2(p)=\lambda\nabla f_1(p)$ for some $\lambda\in\R$. In particular, if $\lambda\leq 0$ the point $p\in M$ is said to be a \emph{critical Pareto point} for $f$. The set of all critical Pareto points of $f$ is denoted by $\mathbb{J}_P(f)$ and is a subset of the Jacobi set $\mathbb{J}(f)$. Obviously, $\mathbb{J}_P(f)$ contains both the critical points of $f_1$ and the critical points of $f_2$.

We assume that
\begin{itemize}
    \item[$(i)$]  No point $p\in M$ exists such that both $\nabla f_1(p)$ and $\nabla f_2(p)$ vanish;
	\item[$(ii)$] $\mathbb{J}(f)$ is a smoothly embedded 1-manifold in $M$ consisting of finitely many components, each one diffeomorphic to a circle;
	\item[$(iii)$] $\mathbb{J}_P(f)$ is a 1-dimensional closed submanifold of $M$, with boundary in $\mathbb{J}(f)$.
\end{itemize}

We consider the set $\mathbb{J}_C(f)$ of \emph{cusp points} of $f$, that is, points of $\mathbb{J}(f)$ at which the restriction of $f$ to $\mathbb{J}(f)$ fails to be an immersion. In other words
$\mathbb{J}_C(f)$ is the subset of $\mathbb{J}(f)$ at which both $\nabla f_1$ and $\nabla f_2$ are orthogonal to $\mathbb{J}(f)$.

We also assume that
\smallskip

$(iv)$ The connected components of $\mathbb{J}_P(f)\setminus\mathbb{J}_C(f)$ are finite in number, each one being diffeomorphic to an interval. With respect to any parameterization of each component, one of $f_1$ and $f_2$ is strictly increasing and the other is strictly decreasing. Each component can meet critical points for $f_1,f_2$ only at its endpoints.
\smallskip

In \cite{Wa75} (see also \cite{EdHa04}) it is proved that the previous  properties $(i),(ii),(iii),(iv)$ are generic in the set of smooth maps from $M$ to $\R^2$.

Property $(iv)$ implies that the connected components of $\mathbb{J}_P(f)\setminus\mathbb{J}_C(f)$ are open, or closed, or semi-open arcs in $M$. Following the notation used in \cite{Wa75}, they will be referred to as \emph{critical intervals} of $f$. If an endpoint $p$ of a critical interval actually belongs to that critical interval and hence is not a cusp point, then it is a critical point for either $f_1$ or $f_2$.
We denote the critical intervals of $f$ by $\alpha_1,\dots,\alpha_r$, and parameterize these arcs arbitrarily, that is, $\alpha_i:I_i\to M$, with $I_i$ equal to $]0,1[$, or $]0,1]$, or $[0,1[$, or $[0,1]$.
Our assumptions also imply that both the set of critical points of $f_1$ and the set of critical points of $f_2$ are finite.

\subsection{The extended Pareto grid}

Our purpose is to establish a formal link between the position of points of $\dgm\left(f_{(a,b)}^*\right)$ for a function $f$ and the intersections of the admissible line $r_{(a,b)}$ with a particular subset of the plane $\R^2$, called the \emph{extended Pareto grid} of $f$, which we will define here.

Let us list the critical points $p_1,\ldots,p_h$ of $f_1$ and the critical points $q_1,\ldots,q_k$ of $f_2$ (our assumption $(i)$ guarantees that $\{p_1,\ldots,p_h\}\cap \{q_1,\ldots,q_k\}=\emptyset$). Consider the following closed half-lines: for each critical point $p_i$ of $f_1$ (resp. each critical point $q_j$ of $f_2$), the half-line $\{(x,y)\in\R^2 : x=f_1(p_i), y\ge f_2(p_i)\}$ (resp. the half-line $\{(x,y)\in\R^2 : x\ge f_1(q_j), y=f_2(q_j)\}$). The extended Pareto grid $\Gamma(f)$ is defined to be the union of $f(\mathbb{J}_P(f))$ with these closed half-lines. The closures of the images of critical intervals of $f$ will be called \emph{proper contours} of $f$ associated with those critical intervals of $f$, while the closed half-lines will be known as \emph{improper contours} of $f$ associated with the corresponding critical points of $f_1$ and $f_2$. We will distinguish between proper contours associated with different critical intervals and between improper contours associated with different critical points, although they can possibly coincide as sets. We observe that every contour is a closed set and the number of contours of $f$ is finite because of property $(iv)$.

Let $\mathcal{S}(f)$ be the set of all points of $\Gamma(f)$ that belong to more than one (proper or improper) contour.
If $\mathcal{S}(f)$ is finite, we say that the \emph{multiplicity} of $P\in \Gamma(f)$ is the greatest $k$ such that
for every $\varepsilon>0$
a line $r_{(a,b)}$ with $(a,b)\in{\mathcal {P}}(\Lambda^+)$ exists, verifying these two properties: $\dagger)$ $r_{(a,b)}$ does not meet $\mathcal{S}(f)$ and $\ddagger)$ the cardinality of $r_{(a,b)}\cap\Gamma(f)\cap B(P,\varepsilon)$ is $k$, where $B(P,\varepsilon)$ is the open ball of center $P$ and radius $\varepsilon$ with respect to the Euclidean distance.
In other words, the multiplicity of $P\in \Gamma(f)$ is the maximum $k$ such that we can find a line with positive slope that does not touch $\mathcal{S}(f)$ and contains $k$ points of the extended Pareto grid that are arbitrarily close to $P$.

Under the assumption that $\mathcal{S}(f)$ is finite, let $\mathcal{D}(f)$ be the set of all points $p\in \Gamma(f)$ that have multiplicity strictly greater than $1$.
We observe that $\mathcal{D}(f)\subseteq \mathcal{S}(f)$.
Each connected component of $\Gamma(f)\setminus\mathcal{D}(f)$ will be called a \emph{contour-arc} of $f$. Therefore, the contour-arcs do not contain their endpoints.

A visual intuition is given by Figure~\ref{figgrid3}, showing the extended Pareto grid of the function $f$ taking each point $p$ of the torus in Figure~\ref{figtoro1} to the pair $f(p)=(x(p),z(p))$. The images of the critical intervals are in red, the vertical half-lines with abscissa equal to a critical value of $f_1$ are in purple, and the horizontal half-lines with ordinate equal to a critical value of $f_2$ are in orange. The extended Pareto grid $\Gamma(f)$ contains the red, purple and orange points. The highlighted red points are endpoints of contours. A blue admissible line $r_{(a,b)}$ that does not meet $\mathcal{S}(f)$ is also represented. The black point belongs to $\mathcal{D}(f)$, since we can find a line with positive slope which does not touch $\mathcal{S}(f)$ and contains exactly $2$ points of the extended Pareto grid that are arbitrarily close to $P$ (see the green points in the figure). The circled point is an example of point of $\mathcal{S}(f)$ that is not multiple.

\begin{figure}
\begin{center}
\begin{tikzpicture}[scale=1.0]
\node at (0,0) {\includegraphics[width=0.4\textwidth,angle=90]{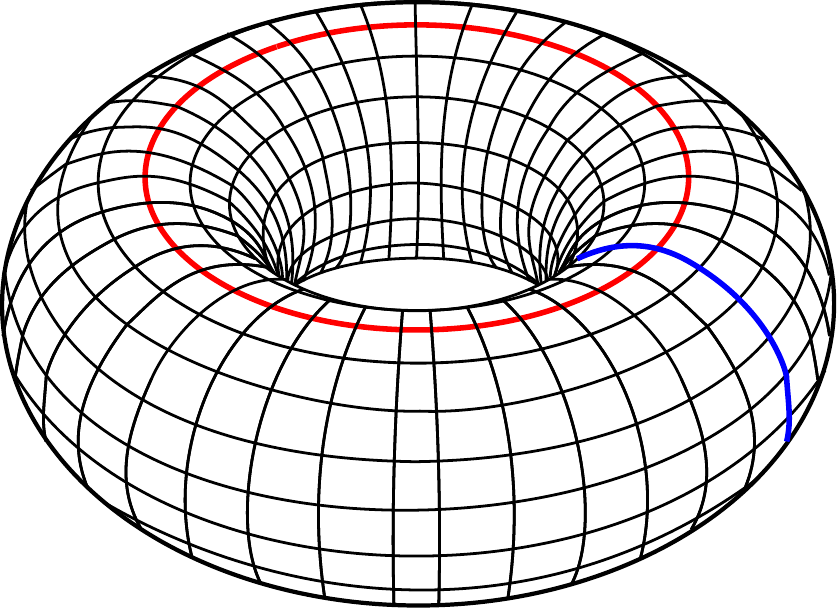}};
\draw [line width=0.8] (-0.15,-0.03) -- (-0.15,0.77);
\draw [dotted,line width=0.8] (-0.15,0.77) -- (-0.15,2.50);
\draw [->,line width=0.8] (-0.15,2.50) -- (-0.15,3.1);
\draw [line width=0.8] (-0.15,-0.03) -- (0.03,-0.15);
\draw [dotted,line width=0.8] (0.03,-0.15) -- (1.40,-1.03);
\draw [->,line width=0.8] (1.40,-1.03) -- (2.0,-1.42);


\draw (-0.35,2.9) node {$z$};
\draw (1.8,-1.55) node {$x$};
\end{tikzpicture}
\end{center}\caption{The torus endowed with the filtering function $f(p):=(x(p),z(p))$.}\label{figtoro1}
\end{figure}
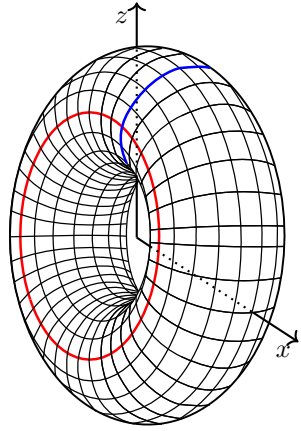

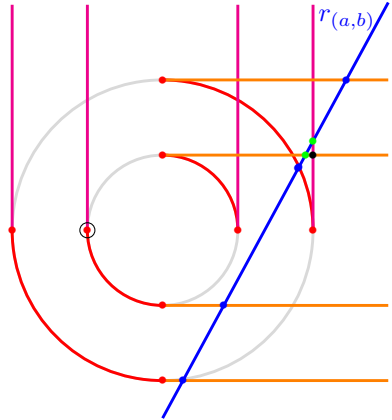
\begin{figure}
\begin{center}

\begin{tikzpicture}[scale=1.0]
\draw [gray!30,line width=1.1] (0,0) circle (1);
\draw [gray!30,line width=1.1] (0,0) circle (2);
\draw [red,line width=1.1] (1,0) arc (0:90:1);
\draw [red,line width=1.1] (2,0) arc (0:90:2);
\draw [red,line width=1.1] (-1,0) arc (180:270:1);
\draw [red,line width=1.1] (-2,0) arc (180:270:2);

\foreach \x in {-2,-1,1,2}
{
	\draw [magenta,line width=1.1] (\x,0) -- (\x,3);
	\draw [orange,line width=1.1] (0,\x) -- (3,\x);

	\draw [red] (\x,0) node {{\tiny $\bullet$}};
	\draw [red] (0,\x) node {{\tiny $\bullet$}};
}

\draw [blue,line width=1.1] (0,-2.5) -- (3,3.03);
\draw [blue] (2.45,2.8) node {$r_{(a,b)}$};

\draw [blue] (0.27,-2) node {{\tiny $\bullet$}};
\draw [blue] (0.81,-1) node {{\tiny $\bullet$}};
\draw [blue] (1.8,0.82) node {{\tiny $\bullet$}};
\draw [blue] (1.81,0.835) node {{\tiny $\bullet$}};
\draw [green] (1.9,1) node {{\tiny $\bullet$}};
\draw [green] (2,1.18) node {{\tiny $\bullet$}};
\draw [black] (2,1) node {{\tiny $\bullet$}};
\draw [blue] (2.44,2) node {{\tiny $\bullet$}};
\draw (-1,-0) circle (0.1cm);
\end{tikzpicture}
\end{center}\caption{The extended Pareto grid for the torus in Figure~\ref{figtoro1}, endowed with the filtering function $f(p):=(x(p),z(p))$.}\label{figgrid3}
\end{figure}

\subsection{Assumptions about the extended Pareto grid}

We recall that, by definition, a pair $(a,b)\in \, ]0,1[\times\R$ is  \emph{singular} for $f$ if and only if the set  $\dgm\left(f_{(a,b)}^*\right)\setminus \{\Delta\}$ contains at least one point having multiplicity strictly greater than 1. A pair $(a,b)$ that is not singular is called \emph{regular}.

\begin{definition}\label{defnassumptions}
We say that the function $f:M\to\R^2$ verifying the properties $(i)-(iv)$ is \emph{normal} if the following statements also hold:
\begin{enumerate}
\item \label{defnassumptionsnummult} The set $\mathcal{S}(f)$ is finite;
\item \label{defnassumptionsmultdoub} Each point of the set $\mathcal{D}(f)$ is double;
\item \label{defnassumptionsnoline} No line $r_{(a,b)}$ exists containing more than two points of $\mathcal{D}(f)$;
\item \label{defnassumptionscontourarcs}
Every contour-arc $\gamma$ of $f$ is associated with a pair $(d(\gamma),s(\gamma))\in\Z\times\{-1,1\}$
such that at each point $(u,v)$ of $\gamma$ the following properties hold for every small enough $\varepsilon>0$, when $i^k_*$ is
the map $H_k(M_{(u-\varepsilon,v-\varepsilon)})\to H_k(M_{(u+\varepsilon,v+\varepsilon)})$ induced by the inclusion $M_{(u-\varepsilon,v-\varepsilon)}\hookrightarrow M_{(u+\varepsilon,v+\varepsilon)}$:
\begin{itemize}
  \item If $k\neq d(\gamma)$, $i^k_*$ is an isomorphism;
  \item If $k= d(\gamma)$ and $s(\gamma)=1$, $i^k_*$ is injective and $\textrm{rank} \left(H_k(M_{(u+\varepsilon,v+\varepsilon)})\right)=
      \textrm{rank} \left(H_k(M_{(u-\varepsilon,v-\varepsilon)})\right)+1$;
  \item If $k= d(\gamma)$ and $s(\gamma)=-1$, $i^k_*$ is surjective and $\textrm{rank} \left(H_k(M_{(u+\varepsilon,v+\varepsilon)})\right)=
      \textrm{rank} \left(H_k(M_{(u-\varepsilon,v-\varepsilon)})\right)-1$.
\end{itemize}

\end{enumerate}
\end{definition}

\begin{remark}\label{remcontour}
It is not difficult to prove that in Property~(\ref{defnassumptionscontourarcs}) of Definition~\ref{defnassumptions} the two groups $H_k(M_{(u-\varepsilon,v-\varepsilon)})$ and $H_k(M_{(u+\varepsilon,v+\varepsilon)})$ can be replaced by the groups $H_k(M_{(u-a\varepsilon,v-(1-a)\varepsilon)})$, $H_k(M_{(u+a\varepsilon,v+(1-a)\varepsilon)})$
for any fixed $a\in\,]0,1[$ without changing the concept of normal function. In plain words, Property~(\ref{defnassumptionscontourarcs}) guarantees that the passage across a contour-arc $\gamma$ along any direction $(a,1-a)$ just creates ($s(\gamma)=1$) or destroys ($s(\gamma)=-1$) exactly one homological class in degree $d(\gamma)$, without producing any homological change in the other degrees (see Figure~\ref{figgrid4}). According to \cite{CeLa16}, this implies that the multiplicity of the points of each contour-arc is $1$ in degree $d(\gamma)$.
\end{remark}

Figure~\ref{figgrid4} shows the contour-arcs and the set $\mathcal{D}(f)$ (in white) for the function taking each point $p$ of the torus in Figure~\ref{figtoro1} to the pair $f(p)=(x(p),z(p))$. Each of the two {\bf \color{magenta}magenta} contour-arcs corresponds to the birth of a homology class in degree $0$ (i.e. $(d(\gamma),s(\gamma))=(0,1)$). Each of the ten {\bf \color{black}black} contour-arcs corresponds to the birth of a homology class in degree $1$ (i.e. $(d(\gamma),s(\gamma))=(1,1)$). Each of the two {\bf \color{blue}blue} contour-arcs corresponds to the birth of a homology class in degree $2$ (i.e. $(d(\gamma),s(\gamma))=(2,1)$). Each of the two {\bf \color{red}red} contour-arcs corresponds to the death of a homology class in degree $0$ (i.e. $(d(\gamma),s(\gamma))=(0,-1)$). Each of the four {\bf \color{green}green} contour-arcs corresponds to the death of a homology class in degree $1$ (i.e. $(d(\gamma),s(\gamma))=(1,-1)$). Note that the homological event associated with the points of a contour of $f$ can change along the considered contour. This justifies the choice of using the concept of contour-arc instead of the one of contour in Property~(\ref{defnassumptionscontourarcs}).

\begin{figure}
\begin{center}
\begin{tikzpicture}[scale=1.0]
\draw [magenta,line width=1.1] (1,0) arc (0:90:1);
\draw [black,line width=1.1] (2,0) arc (0:90:2);
\draw [black,line width=1.1] (-1,0) arc (180:270:1);
\draw [magenta,line width=1.1] (-2,0) arc (180:270:2);

\draw [red,line width=1.1] (1,0) -- (1,1);
\draw [red,line width=1.1] (0,1) -- (1,1);

\draw [black,line width=1.1] (1,1) -- (3,1);
\draw [black,line width=1.1] (1,1) -- (1,3);

\draw [black,line width=1.1] (-1,0) -- (-1,3);
\draw [magenta,line width=1.1] (-2,0) -- (-2,3);
\draw [black,line width=1.1] (0,-1) -- (3,-1);
\draw [magenta,line width=1.1] (0,-2) -- (3,-2);
\draw [green,line width=1.1] (0,2) -- (2,2);
\draw [green,line width=1.1] (2,0) -- (2,2);
\draw [blue,line width=1.1] (2,2) -- (2,3);
\draw [blue,line width=1.1] (2,2) -- (3,2);

\draw [line width=0.1,black,fill=white] (1,1.73) circle (0.04);
\draw [line width=0.1,black,fill=white] (1.73,1) circle (0.04);

\draw [line width=0.1,black,fill=white] (1,0) circle (0.04);
\draw [line width=0.1,black,fill=white] (2,0) circle (0.04);
\draw [line width=0.1,black,fill=white] (0,1) circle (0.04);
\draw [line width=0.1,black,fill=white] (0,2) circle (0.04);

\foreach \x in {1,2}
{
	\foreach \y in {1,2}
	{
		\draw [line width=0.1,black,fill=white] (\x,\y) circle (0.04);
	}
}

\end{tikzpicture}
\end{center}\caption{The connected components obtained by deleting the double points (in white) from $\Gamma(f)$ are the contour-arcs for the torus in Figure~\ref{figtoro1}, endowed with the filtering function $f(p):=(x(p),z(p))$. In this example $\Gamma(f)$ contains 20 contour-arcs.
}\label{figgrid4}
\end{figure}
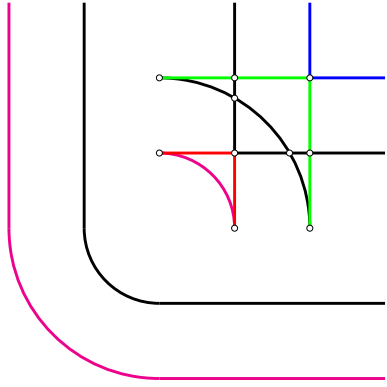

In the rest of this paper we will assume that the function $f:M\to\R^2$ is normal.


\subsection{The Position Theorem}
We recall that
$$f^*_{(a,b)}(p):=\min\{a,1-a\}\cdot \max\left\{\frac{f_1(p)-b}{a},\frac{f_2(p)+b}{1-a}\right\}$$ for
every $p\in M$.
With the concept of extended Pareto grid at hand, we can state and prove the following result, which gives a   necessary condition for $P$ to be a point of $\dgm\left(f_{(a,b)}^*\right)$.

\begin{theorem}[Position Theorem]\label{GT*}
Let $(a,b)\in {\mathcal {P}}(\Lambda^+)$, $P\in\dgm\left(f_{(a,b)}^*\right)\setminus \{\Delta\}$. Then, for each finite coordinate $w$ of $P$  a point $(x,y)\in r_{(a,b)}\cap\Gamma(f)$
exists, such that $w=\frac{\min\{a,1-a\}}{a}\cdot (x-b)=\frac{\min\{a,1-a\}}{1-a}\cdot (y+b)$.
\end{theorem}

\begin{proof}
By applying Theorem~3.2 in  \cite{CeFr14} and recalling that the points of $\dgm\left(f_{(a,b)}^*\right)$ are obtained by multiplying the ones of $\dgm\left(f_{(a,b)}\right)$ by the factor $\min\{a,1-a\}$, we obtain that a point $p\in M$ exists such that one of the following statements holds:
\begin{enumerate}
\item $\nabla f_1(p)=0$, and $w=\frac{\min\{a,1-a\}}{a}\cdot (f_1(p)-b)>\frac{\min\{a,1-a\}}{1-a}\cdot (f_2(p)+b)$;
\item $\nabla f_2(p)=0$, and $w=\frac{\min\{a,1-a\}}{1-a}\cdot (f_2(p)+b)>\frac{\min\{a,1-a\}}{a}\cdot (f_1(p)-b)$;
\item $p\in\mathbb{J}_P(f)$, and $w=\frac{\min\{a,1-a\}}{a}\cdot (f_1(p)-b)=\frac{\min\{a,1-a\}}{1-a}\cdot (f_2(p)+b)$.
\end{enumerate}

Assume that (1) holds, and recall that the admissible line $r_{(a,b)}$ is parameterized by $t$ and has equation $(u(t),v(t))=t\cdot(a,1-a)+(b,-b)$. Looking for the point of $r_{(a,b)}$ whose abscissa is $f_1(p)$, we find  $(x,y):=\left(\frac{aw}{\min\{a,1-a\}}+b,\frac{(1-a)w}{\min\{a,1-a\}}-b\right)$. Since $w>\frac{\min\{a,1-a\}}{1-a}\cdot (f_2(p)+b)$, we have that
$y>f_2(p)$. This means that at $(x,y)$ the line $r_{(a,b)}$ meets the vertical (open) half-line $r: x=f_1(p),y>f_2(p)$, which is part of the extended Pareto grid (recall that, by (1), $p$ is a critical point for $f_1$).
Therefore, $w=\frac{\min\{a,1-a\}}{a}\cdot (x-b)=\frac{\min\{a,1-a\}}{1-a}\cdot (y+b)$ with $(x,y)\in r_{(a,b)}\cap\Gamma(f)$.

We skip the case in which (2) holds, because it is completely analogous to the one just considered.

To conclude the proof, assume now that (3) holds.
We know that the point $(f_1(p),f_2(p))$ belongs to $\Gamma(f)$, because $p\in\mathbb{J}_P(f)$. Given that the admissible line $r_{(a,b)}$ is parameterized by $t$ and has equation $(u(t),v(t))=t\cdot(a,1-a)+(b,-b)$, by taking $t=\frac{w}{\min\{a,1-a\}}$ we have that $(f_1(p),f_2(p))=(u(t),v(t))$ belongs also to $r_{(a,b)}$. By setting $(x,y):=(f_1(p),f_2(p))\in r_{(a,b)}\cap \Gamma(f)$ we get $w=\frac{\min\{a,1-a\}}{a}\cdot (x-b)=\frac{\min\{a,1-a\}}{1-a}\cdot (y+b)$.
This yields the claim.
\end{proof}

The Position Theorem~\ref{GT*} suggests a way to find the possible positions for points of $\dgm\left(f_{(a,b)}^*\right)$. It consists in drawing the extended Pareto grid $\Gamma(f)$ and considering its intersections $(x_1,y_1),\ldots, (x_l,y_l)$ with the admissible line $r_{(a,b)}$. For each point of $\dgm\left(f_{(a,b)}^*\right)\setminus \{\Delta\}$, both its coordinates belong to the set
\begin{equation}\label{coord}
\left\{
\frac{\min\{a,1-a\}}{a}\cdot (x_i-b)=\frac{\min\{a,1-a\}}{1-a}\cdot (y_i+b)\right\}_{1\le i\le l}\cup \{\infty\}.
\end{equation}
Note that when $b<0$ and $|b|$ is sufficiently large, the admissible line $r_{(a,b)}$ may intersect $\Gamma(f)$ only at the vertical half-lines (see line $r_{(a,b')}$ in Figure~\ref{figgrid5}). In this case, $f^*_{(a,b)}:=\frac{\min\{a,1-a\}}{a}\cdot (f_1-b)$, and the values $x_1,\dots,x_l$ in (\ref{coord}) are the critical values of $f_1$. Similarly, when $b>0$ and $|b|$ is large enough, $r_{(a,b)}$ intersects $\Gamma(f)$ only at the horizontal half-lines (see line $r_{(a,b'')}$ in Figure~\ref{figgrid5}). Then $f^*_{(a,b)}:=\frac{\min\{a,1-a\}}{1-a}\cdot (f_2+b)$, and the values $y_1,\dots,y_l$ in (\ref{coord}) are the critical values of $f_2$.

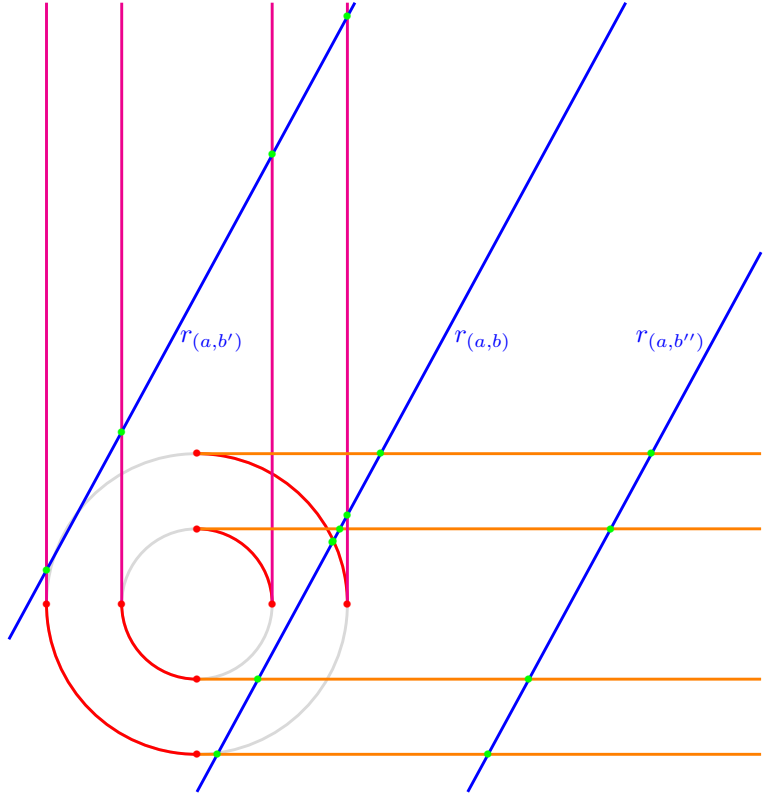
\begin{figure}
\begin{center}
\begin{tikzpicture}[scale=1.0]
\draw [gray!30,line width=1.1] (0,0) circle (1);
\draw [gray!30,line width=1.1] (0,0) circle (2);
\draw [red,line width=1.1] (1,0) arc (0:90:1);
\draw [red,line width=1.1] (2,0) arc (0:90:2);
\draw [red,line width=1.1] (-1,0) arc (180:270:1);
\draw [red,line width=1.1] (-2,0) arc (180:270:2);

\foreach \x in {-2,-1,1,2}
{
	\draw [magenta,line width=1.1] (\x,0) -- (\x,8);
	\draw [orange,line width=1.1] (0,\x) -- (7.5,\x);

	\draw [red] (\x,0) node {{\tiny $\bullet$}};
	\draw [red] (0,\x) node {{\tiny $\bullet$}};
}

\draw [blue,line width=1.1] (-2.5,-0.47) -- (2.1,8);
\draw [blue] (0.2,3.5) node {$r_{(a,b')}$};
\draw [blue,line width=1.1] (0,-2.5) -- (5.7,8);
\draw [blue] (3.8,3.5) node {$r_{(a,b)}$};
\draw [blue,line width=1.1] (3.6,-2.5) -- (7.5,4.68);
\draw [blue] (6.3,3.5) node {$r_{(a,b'')}$};

\draw [green] (-2,0.45) node {{\tiny $\bullet$}};
\draw [green] (-1,2.29) node {{\tiny $\bullet$}};
\draw [green] (1,5.98) node {{\tiny $\bullet$}};
\draw [green] (2,7.82) node {{\tiny $\bullet$}};

\draw [green] (0.27,-2) node {{\tiny $\bullet$}};
\draw [green] (0.81,-1) node {{\tiny $\bullet$}};
\draw [green] (1.8,0.82) node {{\tiny $\bullet$}};
\draw [green] (1.81,0.835) node {{\tiny $\bullet$}};
\draw [green] (1.9,1) node {{\tiny $\bullet$}};
\draw [green] (2,1.18) node {{\tiny $\bullet$}};
\draw [green] (2.44,2) node {{\tiny $\bullet$}};

\draw [green] (3.87,-2) node {{\tiny $\bullet$}};
\draw [green] (4.41,-1) node {{\tiny $\bullet$}};
\draw [green] (5.5,1) node {{\tiny $\bullet$}};
\draw [green] (6.04,2) node {{\tiny $\bullet$}};
\end{tikzpicture}
\end{center}\caption{When $b<0$ and $|b|$ is large enough, the line $r_{(a,b)}$ intersects only the vertical half-lines
in the extended Pareto grid. When $b>0$ and $|b|$ is large enough, the line $r_{(a,b)}$ intersects only the horizontal half-lines
in the extended Pareto grid.}
\label{figgrid5}
\end{figure}

\subsection{Pairing of contour-arcs}
\label{pairing}

Two contour-arcs $\gamma_1,\gamma_2$ for $f$ are called \emph{paired to each other with respect to $r_{(a,b)}$} if $r_{(a,b)}$ meets both $\gamma_1$ and $\gamma_2$ at two respective points $(x_1,y_1)$, $(x_2,y_2)$, and $\frac{\min\{a,1-a\}}{a}\cdot \left(x_1-b,x_2-b\right)\in \dgm\left(f_{(a,b)}^*\right)$ (or, equivalently, $\frac{\min\{a,1-a\}}{1-a}\cdot \left(y_1+b,y_2+b\right)\in \dgm\left(f_{(a,b)}^*\right)$).
In plain words, two contour-arcs are paired to each other with respect to $r_{(a,b)}$ if one of them is associated with the birth of a homological class in the filtration given by $f_{(a,b)}^*$ and the other is associated with the death of the same homological class in the same filtration.
By applying the Position Theorem~\ref{GT*} and the Stability Theorem~\ref{stability},
it is easy to check that if the contour-arcs $\gamma_1,\gamma_2$ are paired to each other with respect to $r_{(a,b)}$, and $r_{(a',b')}$ is an admissible line meeting both $\gamma_1$ and $\gamma_2$ at two respective points $(x'_1,y'_1)$, $(x'_2,y'_2)$, then $\gamma_1,\gamma_2$ are paired to each other with respect to $r_{(a',b')}$ as well. We underline that each contour-arc can be paired to different contour-arcs with respect to different admissible lines.

\subsection{Localization of singular pairs by the Position Theorem}
\label{locsing}

The Position Theorem allows us to deduce where singular pairs can be in ${\mathcal {P}}(\Lambda^+)$.

\begin{proposition}\label{propldp}
Let $\left(\bar a,\bar b\right)\in {\mathcal {P}}(\Lambda^+)$ be a singular pair for $f$. If $\dgm\left(f_{\left(\bar a,\bar b\right)}^*\right)$ contains a proper multiple point, then $r_{\left(\bar a,\bar b\right)}$ contains two points of $\mathcal{D}(f)$.
If $\dgm\left(f_{\left(\bar a,\bar b\right)}^*\right)$ contains an improper multiple point, then $r_{\left(\bar a,\bar b\right)}$ contains at least one point of $\mathcal{D}(f)$.
\end{proposition}

\begin{proof}
Let us first assume that $\dgm\left(f_{\left(\bar a,\bar b\right)}^*\right)$ contains a proper multiple point $(\bar u,\bar v)$.
We can find a line $r_{(a',b')}$ that is arbitrarily close to $r_{\left(\bar a,\bar b\right)}$ and does not meet $\mathcal{S}(f)$.
Because of the Stability Theorem~\ref{stability}, $\dgm\left(f_{(a',b')}^*\right)$ must contain two proper points arbitrarily close to each other. Therefore, the Position Theorem~\ref{GT*} and the definition of the set $\mathcal{D}(f)$ imply that $r_{\left(\bar a,\bar b\right)}$ contains two double points of $\Gamma(f)$.

Let us now assume that $\dgm\left(f_{\left(\bar a,\bar b\right)}^*\right)$ contains an improper multiple point.
Also in this case let us consider a line $r_{(a',b')}$ that is arbitrarily close to $r_{\left(\bar a,\bar b\right)}$ and does not meet $\mathcal{S}(f)$.
Because of the Stability Theorem~\ref{stability}, $\dgm\left(f_{(a',b')}^*\right)$ must contain two improper points arbitrarily close to each other. Therefore, the Position Theorem~\ref{GT*} and the definition of the set $\mathcal{D}(f)$ imply that $r_{\left(\bar a,\bar b\right)}$ contains at least one double point of $\Gamma(f)$.
\end{proof}

Figure~\ref{figdp} illustrates the statement in Proposition~\ref{propldp} in the case of a proper double point of $\dgm\left(f_{\left(\bar a,\bar b\right)}^*\right)$.

\begin{figure}
\begin{center}
\begin{tikzpicture}[scale=1.0]
\draw [line width=1.2] (0,0) -- (6,5);
\draw (5.8,4.3) node {$r_{\left(\bar a,\bar b\right)}$};
\draw [red,line width=1.2] (2,0.28) arc (30:70:3);
\draw [red,line width=1.2] (1.52,0) arc (0:40:3);
\draw [line width=0.1,black,fill=white] (1.31,1.09) circle (0.04);

\draw [red,line width=1.2] (5.5,3.20) arc (30:70:3);
\draw [red,line width=1.2] (5.02,2.92) arc (0:40:3);
\draw [line width=0.1,black,fill=white] (4.81,4.01) circle (0.04);
\end{tikzpicture}
\end{center}\caption{A line $r_{\left(\bar a,\bar b\right)}$ associated with a singular pair $\left(\bar a,\bar b\right)\in{\mathcal {P}}(\Lambda^+)$, in case $\dgm\left(f_{\left(\bar a,\bar b\right)}^*\right)$ contains a proper double point. Parts of four contours (split in eight proper contour-arcs) are displayed in red.}
\label{figdp}
\end{figure}
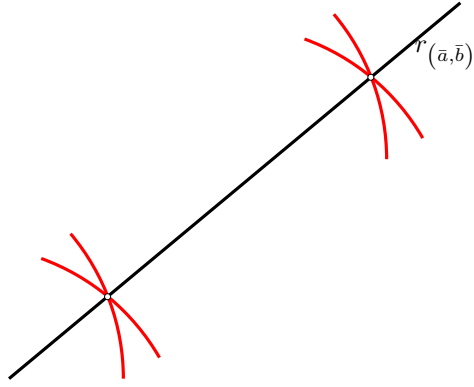

We conclude this subsection by giving some results that will be of use in the paper.

\begin{corollary}\label{corfiniteness}
The multiset $\dgm\left(f_{(a,b)}^*\right)\setminus \{\Delta\}$ is finite for every $(a,b)\in{\mathcal {P}}(\Lambda^+)$.
\end{corollary}

\begin{proof}
The statement immediately follows from the Position Theorem, the properties in Definition~\ref{defnassumptions} and
the Stability Theorem~\ref{stability}.
\end{proof}

\begin{corollary}\label{corsp}
The following statements hold:
\begin{enumerate}
\item The set of all pairs $(a,b)\in{\mathcal {P}}(\Lambda^+)$ such that $\dgm\left(f_{(a,b)}^*\right)$ contains a proper multiple point is finite;
\item The set of all pairs $(a,b)\in{\mathcal {P}}(\Lambda^+)$ such that $\dgm\left(f_{(a,b)}^*\right)$ contains an improper multiple point is a finite union of open segments joining the line $a=0$ to the line $a=1$.
\end{enumerate}
\end{corollary}

\begin{proof}
Statement $(1)$ follows from Proposition~\ref{propldp}, Property (\ref{defnassumptionsnummult}) in Definition~\ref{defnassumptions} and the inclusion $\mathcal{D}(f)\subseteq\mathcal{S}(f)$. As for Statement $(2)$, let us assume that the set $\dgm\left(f_{\left(\bar a,\bar b\right)}^*\right)$ contains an improper multiple point. Because of Proposition~\ref{propldp}, $r_{\left(\bar a,\bar b\right)}$ contains a double point $D\in\Gamma(f)$. It is easy to check that the sheaf of positive slope lines passing through $D$ corresponds to a segment $S_D$ joining the line $a=0$ to the line $a=1$ in ${\mathcal {P}}(\Lambda^+)$.
%
The definition of multiplicity of points in $\Gamma(f)$ and Remark~\ref{remcontour} imply that $\dgm\left(f_{(a,b)}^*\right)$ contains an improper multiple point for every $(a,b)\in S_D$.
\end{proof}

\begin{corollary}\label{cordpp}
For every $(a,b)\in {\mathcal {P}}(\Lambda^+)$, at most one proper point of $\dgm\left(f_{(a,b)}^*\right)$ has multiplicity greater than $1$. If such a point exists, its multiplicity is $2$.
\end{corollary}

\begin{proof}
The Position Theorem~\ref{GT*} and the properties in Definition~\ref{defnassumptions} guarantee
that if $r_{(a',b')}$ is an admissible line that is close to $r_{(a,b)}$ and does not touch $\mathcal{D}(f)$, then it is not possible to find more than one pair $\left((u_1,v_1),(u_2,v_2)\right)$ of proper points in  $\dgm\left(f^*_{(a',b')}\right)$ such that $(u_1,v_1)$ is arbitrarily close to $(u_2,v_2)$.
Our statements follow from the Stability Theorem~\ref{stability}.
\end{proof}

\begin{corollary}\label{cordip}
For every $(a,b)\in {\mathcal {P}}(\Lambda^+)$, at most two improper points of $\dgm\left(f_{(a,b)}^*\right)$ have multiplicity greater than $1$. If such points exist, their multiplicities are $2$.
\end{corollary}

\begin{proof}
The Position Theorem~\ref{GT*} and the properties in Definition~\ref{defnassumptions} guarantee
that if $r_{(a',b')}$ is an admissible line that is close to $r_{(a,b)}$ and does not touch $\mathcal{D}(f)$, then
it is not possible to find more than two pairs $\left((u_1,\infty),(u_2,\infty)\right),\left((u'_1,\infty),(u'_2,\infty)\right)$ of improper points in  $\dgm\left(f^*_{(a',b')}\right)$ such that $u_1$ is arbitrarily close to $u_2$ and
$u'_1$ is arbitrarily close to $u'_2$.
Our statements follow from the Stability Theorem~\ref{stability}.
\end{proof}

The results proved in this subsection are illustrated by the following example.

\begin{example}\label{excor}
Let us consider the union $M'$ of two disjoint spherical surfaces in $\R^3$, having the extended Pareto grid represented in Figure~\ref{figcor} with respect to the filtering function $f$ that takes each point $p\in M'$ to the pair $f(p)=(x(p),z(p))$.
Let us consider the admissible line $r_{\left(\bar a,\bar b\right)}$ meeting the double points $A$ and $B$. It is easy to check that $\left(\bar a,\bar b\right)$ is a singular pair of ${\mathcal {P}}(\Lambda^+)$ in degree $1$.
Moreover, if the line $r_{(a,b)}$ contains the double point $C$, then $(a,b)$ is a singular pair of ${\mathcal {P}}(\Lambda^+)$ in degree $0$. Furthermore, if the line $r_{(a,b)}$ contains the double point $D$, then $(a,b)$ is a singular pair of ${\mathcal {P}}(\Lambda^+)$ in degree $2$.
\end{example}

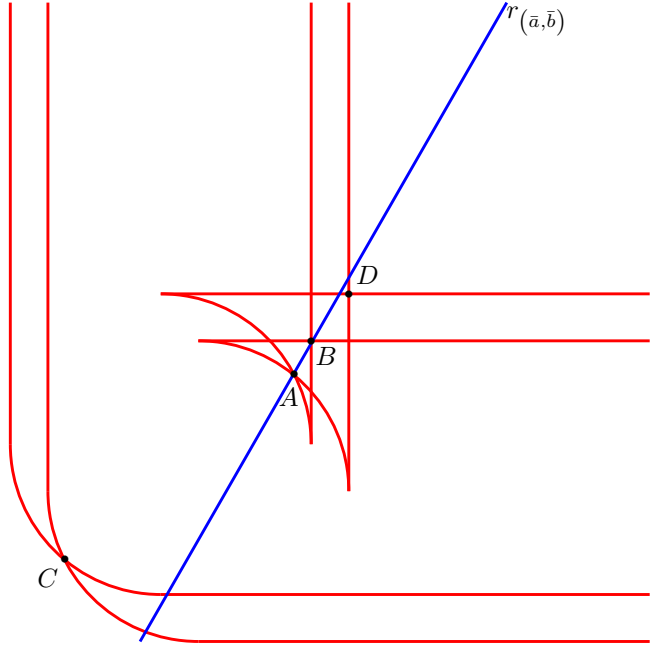
\begin{figure}
\begin{center}
\begin{tikzpicture}[scale=0.5]

\draw [red,line width=1.1] (3,1.25) arc (0:90:4);
\draw [red,line width=1.1] (-5,1.25) arc (180:270:4);

\draw [red,line width=1.1] (-1,-2.75) -- (12,-2.75);
\draw [red,line width=1.1] (-5,1.25) -- (-5,13);

\draw [red,line width=1.1] (-1,5.25) -- (12,5.25);
\draw [red,line width=1.1] (3,1.25) -- (3,13);

\draw [red,line width=1.1] (4,0) arc (0:90:4);
\draw [red,line width=1.1] (-4,0) arc (180:270:4);

\draw [red,line width=1.1] (0,-4) -- (12,-4);
\draw [red,line width=1.1] (-4,0) -- (-4,13);

\draw [red,line width=1.1] (0,4) -- (12,4);
\draw [red,line width=1.1] (4,0) -- (4,13);

\draw [blue,line width=1.1] (-1.55,-4) -- (8.2,13);

\draw [black] (2.55,3.1) node {{\tiny $\bullet$}};
\draw [black] (3,4) node {{\tiny $\bullet$}};
\draw [black] (-3.55,-1.8) node {{\tiny $\bullet$}};
\draw [black] (4,5.25) node {{\tiny $\bullet$}};

\draw (9,12.5) node {$r_{\left(\bar a,\bar b\right)}$};

\draw (2.4,2.5) node {$A$};

\draw (3.4,3.6) node {$B$};

\draw (-4,-2.3) node {$C$};

\draw (4.5,5.75) node {$D$};

\end{tikzpicture}
\end{center}\caption{
The extended Pareto grid of the manifold $M'$ described in Example~\ref{excor}, with respect to the filtering function $f$ that takes each point $p\in M'$ to the pair $f(p)=(x(p),z(p))$. The blue line corresponds to a singular pair of ${\mathcal {P}}(\Lambda^+)$ in degree $1$.
}\label{figcor}
\end{figure}

\subsection{Creation and destruction of points in $\dgm\left(f_{(a,b)}^*\right)$ varying $(a,b)$}
\label{C&D}

From the Position Theorem~\ref{GT*} the next Proposition~\ref{propDelta} follows, allowing us to deduce where in $\Delta$ a point  of $\dgm\left(f_{(a,b)}^*\right)$ can be created or destroyed. We start by giving two definitions.

\begin{definition}
Let $\gamma_1,\gamma_2$ be two contour-arcs paired to each other (with respect to any line $r_{(a,b)}$ that meets both of them). If $\gamma_1$ and $\gamma_2$ have a common endpoint $(\bar x,\bar y)$, it is known as an \emph{annihilation crossing} for $f$, associated with the contour-arcs $\gamma_1,\gamma_2$.
\end{definition}

By definition, each annihilation crossing for $f$ belongs to the set $\mathcal{D}(f)$.
The set of all annihilation crossings for $f$ will be denoted by the symbol $\mathcal{A}(f)$.

\begin{definition}
Let $(a,b)\in{\mathcal {P}}(\Lambda^+)$ and $(\bar u,\bar u)\in\Delta$.
If for every $\delta>0$ a pair $(a',b')\in{\mathcal {P}}(\Lambda^+)$ and a point $(u',v')\in \dgm\left(f^*_{(a',b')}\right)\setminus \{\Delta\}$ exist with
$|a-a'|,|b-b'|< \delta$ and $|\bar u-u'|,|\bar u-v'|<\delta$, then $(\bar u,\bar u)$ will be called an \emph{annihilation point} at $(a,b)$ for $f$.
\end{definition}

In plain words, the annihilation points at $(a,b)$ are the locations on the diagonal $\Delta$ at which the points of the persistence diagram $\dgm\left(f_{(a,b)}^*\right)$ can appear and disappear.

\begin{proposition}\label{propDelta}
A point $(\bar x,\bar y)\in\R^2$ is an annihilation crossing for $f$ if and only if for every
$r_{(a,b)}$ containing $(\bar x,\bar y)$, the point $(\bar u,\bar u)$ with $\bar u=\frac{\min\{a,1-a\}}{a}\cdot (\bar x-b)=\frac{\min\{a,1-a\}}{1-a}\cdot (\bar y+b)$ is an annihilation point at $(a,b)$ for $f$.
\end{proposition}

\begin{proof}
This is a direct consequence of the Position Theorem~\ref{GT*} and the Stability Theorem~\ref{stability}.
\end{proof}

This corollary immediately follows.

\begin{corollary}\label{corDelta}
Let $(a(t),b(t))$ be a continuous curve in ${\mathcal {P}}(\Lambda^+)$ such that the distance between
$\dgm\left(f^*_{(a(t),b(t))}\right)\setminus \{\Delta\}$ and $(\bar u,\bar u)\in\Delta$ tends to $0$ for $t\to \bar t$.
Then an annihilation crossing $(\bar x,\bar y)\in \mathcal{A}(f)$ exists with $\bar u=\frac{\min\{a(\bar t),1-a(\bar t)\}}{a(\bar t)}\cdot (\bar x-b(\bar t))=\frac{\min\{a(\bar t),1-a(\bar t)\}}{1-a(\bar t)}\cdot (\bar y+b(\bar t))$.
\end{corollary}

In plain words, the previous result shows that points of $\dgm\left(f^*_{(a(t),b(t))}\right)$ can be created or destroyed only when the line $r_{(a(t),b(t))}$ meets an annihilation crossing $(\bar x,\bar y)\in \mathcal{A}(f)$ (see Figure~\ref{figdestruction}). Then the creation or destruction happens at the annihilation point $(\bar u,\bar u)$ at
$\left(a,b\right)$ with $\bar u=\frac{\min\{a(\bar t),1-a(\bar t)\}}{a(\bar t)}\cdot (\bar x-b(\bar t))$.

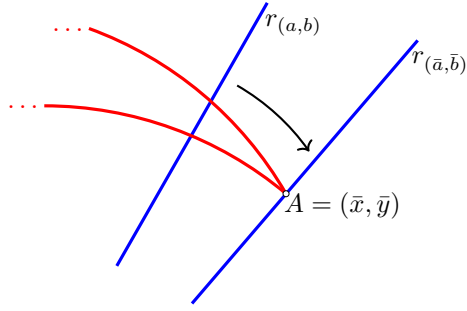
\begin{figure}
\begin{center}
\begin{tikzpicture}[scale=1.0]
\draw [blue,line width=1.2] (2,0.5) -- (4,4);
\draw [blue,line width=1.2] (3,0) -- (6,3.5);
\draw (4.35,3.7) node {$r_{(a,b)}$};
\draw (6.3,3.2) node {$r_{(\bar{a},\bar{b})}$};

\draw [red,line width=1.2] (4.25,1.46) arc (30:70:5);
\draw [red,line width=1.2] (4.25,1.46) arc (50:90:5);

\draw [line width=0.1,black,fill=white] (4.25,1.46) circle (0.04);
\draw (5,1.3) node {$A=(\bar x,\bar y)$};

\draw [red] (0.75,2.62) node {$\ldots$};
\draw [red] (1.35,3.65) node {$\ldots$};

\draw [->,line width=0.8] (3.6,2.9) arc (60:35:3);
\end{tikzpicture}
\end{center}\caption{When $(a,b)$ moves towards $(\bar{a},\bar{b})$ and, correspondingly, the line $r_{(a,b)}$ moves and meets an annihilation crossing $A\in \mathcal{A}(f)$, a point of $\dgm\left(f_{(a,b)}^*\right)$ reaches the diagonal $\Delta$ at an annihilation point $(\bar u,\bar u)$ and disappears. By reversing the movement of $r_{(a,b)}$ we get the birth of a point of $\dgm\left(f_{(a,b)}^*\right)$. Parts of two contour-arcs of $f$ are displayed in red.}
\label{figdestruction}
\end{figure}

\subsection{Choice of the functional set $\mathcal{F}_{U,c}$} Now, in order to proceed we fix a connected open subset $U$ of ${\mathcal {P}}(\Lambda^+)$, choose a $c>0$ and define $\mathcal{F}_{U,c}$ as the set of all normal functions $f:M\to\R^2$ such that $\mathrm{Reg}(f)\supseteq U$ and the distance $d(P,Q)$ between any two distinct points $P,Q$ of $\dgm\left(f_{(a,b)}^*\right)$ is strictly greater than $2c>0$ for every $(a,b)\in U$. Incidentally, we observe that for any $(a,b)\in U$ and any $f\in \mathcal{F}_{U,c}$, the sup-norm distance between the points in $\dgm\left(f_{(a,b)}^*\right)\setminus \{\Delta\}$ and the diagonal $\Delta$ is strictly larger than $2c$.
It follows from the previous Subsection~\ref{C&D} that if $(a,b)\in U$ then $r_{(a,b)}$ cannot contain annihilation crossings for $f$.
We also observe that the distance between the sets $U$ and $\mathrm{Sing}(f)$ is positive, because of the Stability Theorem~\ref{stability}.
The definition of our coherent matching distance will depend on the choice of this set $U$.

In the following we will assume that the functions $f,g$ belong to the set $\mathcal{F}_{U,c}$.
This assumption guarantees that the points in the persistence diagrams $\dgm\left(f_{(a,b)}^*\right)$ are far enough from each other, and that the same holds for $\dgm\left(g_{(a,b)}^*\right)$.

\section{The coherent 2-dimensional matching distance}\label{CMD}

In Sections~\ref{2DS} and~\ref{preparing} we have introduced some machinery in order to define and manage the coherent 2-dimensional matching distance between 2D persistence diagrams. Our next step is the definition of the coherent 2-dimensional matching distance~\cite{CeEtFr16}.

The existence of monodromy implies that each loop in $\mathrm{Reg}(f)$ induces a permutation on $\dgm\left(f_{(a,b)}^*\right)$. In other words, it is not possible to establish which point in $\dgm\left(f_{(a,b)}^*\right)$ corresponds to which point in $\dgm\left(f^*_{(a',b')}\right)$ for $(a,b)\neq (a',b')$, since the answer depends on the path that is considered from $(a,b)$ to $(a',b')$ in the parameter space $\mathrm{Reg}(f)$. As a consequence, different paths going from $(a,b)$ to $(a',b')$ might produce different results while ``transporting'' a matching $\sigma_{(a,b)}:\dgm\left(f_{(a,b)}^*\right)\to \dgm\left(g_{(a,b)}^*\right)$ to another point $(a',b')\in {\mathcal {P}}(\Lambda^+)$. Nevertheless, we will see that it is still possible to define a notion of coherent 2-dimensional matching distance. These ideas will be formalized in the upcoming sections.

\subsection{Transporting a matching along a path}\label{T}
Using the Stability Theorem~\ref{stability}, we will clarify the notion of transport of cornerpoints, that is, we will follow the movement of each point $P\in\dgm\left(f_{(a,b)}^*\right)$ when $(a,b)$ varies in $\mathcal{P}(\Lambda^+)$.
First, we need to specify the concept of transporting a point $X\in\dgm\left(f_{(a(0),b(0))}^*\right)$ along a path $(a(t),b(t))$ in $U$, for $t\in [0,1]$.
We recall that $\Delta:=\{(u,v)\in\R\times\R:u=v\}$, $\Delta^*:=\{(u,v)\in \R\times (\R\cup \{\infty\}):u<v\}$, and we are assuming $f \in\mathcal{F}_{U,c}$.

\begin{definition}[Induced path]\label{admPath}
A continuous path $P:[0,1]\to\Delta^*\cup \{\Delta\}$ is said to be \emph{induced by the path $\pi:[0,1]\to U$} if $P(\tau)\in \dgm\left(f_{\pi(\tau)}^*\right)$ for every $\tau\in [0,1]$.
\end{definition}

\begin{proposition}\label{uniquePath}
Let $\pi:[0, 1]\to U$ be a continuous path. For every point $X\in\dgm\left(f_{\pi(0)}^*\right)$, a unique continuous path $P:[0,1]\to\Delta^*\cup \{\Delta\}$ induced by $\pi$ exists, such that
$P(0) = X$. If $X= \Delta$ then $P([0,1])=\{\Delta\}$, otherwise $P([0,1])\subseteq \Delta^*$.
\end{proposition}

\begin{proof}\label{proofprop_uniquepath}
If $X=\Delta$ our statement trivially holds, since we can set $P(\tau):=\Delta$ for every $\tau\in[0,1]$.
The uniqueness of this path follows from the assumption $f \in\mathcal{F}_{U,c}$.
Therefore, let us assume that $X\neq\Delta$.

Let $\Theta$ be the set of all values $\theta\in[0,1]$ such that there exists exactly one
continuous path $P_\theta:[0,\theta]\to\Delta^*$ satisfying the equalities $P_\theta(0)=X$ and $P_\theta(\tau)\in\dgm\left(f_{\pi(\tau)}^*\right)\setminus\{\Delta\}$ for every $\tau\in [0,\theta]$. Obviously, $0\in\Theta$.
Furthermore, if $\theta_1,\theta_2\in\Theta$ and $\theta_1\le\theta_2$ then $P_{\theta_1}$ and $P_{\theta_2}$ coincide on $[0,\theta_1]$. Set $\bar\theta:=\sup \Theta$.

First of all, the fact that $\lim_{\theta\to\bar\theta^-}\|f_{\pi(\theta)}^*-f_{\pi(\bar\theta)}^*\|_\infty=0$, the Stability Theorem~\ref{stability} and the assumption $f\in \mathcal{F}_{U,c}$ imply that $\lim_{\theta\to\bar\theta^-} P_\theta(\theta)$ exists and belongs to $\dgm\left(f_{\pi(\bar\theta)}^*\right)\setminus\{\Delta\}$.
By setting $P_{\bar\theta}(\tau)=P_\tau(\tau)$ for $\tau\in[0,\bar\theta[$ and $P_{\bar\theta}(\bar\theta)=\lim_{\theta\to\bar\theta^-} P_\theta(\theta)$, we get a continuous path $P_{\bar\theta}:[0,\bar\theta]\to\Delta^*$ with $P_{\bar\theta}(0)=X$ and $P_{\bar\theta}(\tau)\in\dgm\left(f_{\pi(\tau)}^*\right)\setminus\{\Delta\}$ for every $\tau\in [0,\bar\theta]$, so proving that $\bar\theta\in\Theta$.

We will prove by contradiction that $\bar\theta=1$. Therefore, suppose that $\bar\theta<1$, choose a small $\eta>0$ and consider $\tau\in [\bar\theta,\bar\theta+\eta]\subseteq [\bar\theta,1]$. The Stability Theorem~\ref{stability} and the assumptions that $\pi(\bar\theta)\in U$ and $f\in\mathcal{F}_{U,c}$ imply the existence of an $\eps>0$ such that if $\eta$ is small enough, then $\dgm\left(f_{\pi(\tau)}^*\right)\cap \Delta^*$ contains exactly one point $Q(\tau)$ at a distance less than or equal to $\eps$ from $P_{\bar\theta}(\bar\theta)$. We can then consider the path $P_{\bar\theta+\eta}:[0,\bar\theta+\eta]\to\Delta^*$ defined by setting $P_{\bar\theta+\eta}(\tau):=P_{\bar\theta}(\tau)$ for $\tau\in [0, \bar\theta]$, and
$P_{\bar\theta+\eta}(\tau):=Q(\tau)$ for $\tau\in\ ]\bar\theta, \bar\theta+\eta]$. Once again, by the Stability Theorem~\ref{stability}, it is easy to prove that the path $P_{\bar\theta+\eta}$ is continuous. Moreover no other continuous path $P'_{\bar\theta+\eta}:[0,\bar\theta+\eta]\to\Delta^*$ can exist, verifying the property $P'_{\bar\theta+\eta}(\tau)\in \dgm\left(f_{\pi(\tau)}^*\right)$ for every $\tau\in [0,\bar\theta+\eta]$.
If this were not the case, then a $\theta'\in\ ]\bar\theta,\bar\theta+\eta]$ should exist, such that
$P'_{\bar\theta+\eta}$ differs from $P_{\bar\theta+\eta}$ in any right neighborhood of $\theta'$, while they coincide on $[0,\theta']$. Because of the 1-dimensional Stability Theorem and  the definition of induced path, then $\pi(\theta')$ would be a singular pair belonging to $U$ against the assumption $f\in \mathcal{F}_{U,c}$.

As a consequence, $\bar\theta+\eta$ should belong to $\Theta$, against the definition of $\bar\theta$. This contradiction shows that $\bar\theta=1$, and hence our statement is proved.
\end{proof}

\subsubsection{The definition of transported matching}
\label{deftm}
With reference to the previous Proposition~\ref{uniquePath}, we say that \emph{$\pi$ transports $X$ to $X'=P(1)$ with respect to $f$} and write $T^f_\pi(X)=X'$. We observe that $T^f_\pi$ is a bijection from
$\dgm\left(f_{\pi(0)}^*\right)$ to $\dgm\left(f_{\pi(1)}^*\right)$, whose inverse is the map $T^f_{\pi^{-1}}$, where
$\pi^{-1}$ is the inverse path of $\pi$.
Moreover,  $T^f_\pi\left(\dgm\left(f_{\pi(0)}^*\right)\setminus\{\Delta\}\right)=\dgm\left(f_{\pi(1)}^*\right)\setminus\{\Delta\}$
and $T^f_\pi(\Delta)=\Delta$. In other words, the transport takes points different from $\Delta$ to points different from $\Delta$, and $\Delta$ to $\Delta$.
We now need to define the concept of transporting a matching along a path $\pi: [0,1]\to U$ with $\pi(0) = (a,b)$.  Let $\sigma_{(a,b)}$ be a matching between $\dgm\left(f_{(a,b)}^*\right)$ and $\dgm\left(g_{(a,b)}^*\right)$, with $(a,b)$ an element of $U$, assuming $f,g\in\mathcal{F}_{U,c}$. We can naturally associate to $\sigma_{(a,b)}$ a matching $\sigma_{\pi(1)}:\dgm\left(f_{\pi(1)}^*\right)\to\dgm\left(g_{\pi(1)}^*\right)$. Suppose that $\sigma_{(a,b)}(X)=Y$. We set  $\sigma_{\pi(1)} (X') = Y'$ if and only if $T^f_\pi(X)=X'$ and $T^g_\pi(Y)=Y'$.
We also say that \emph{$\pi$ transports $\sigma_{(a,b)}$ to $\sigma_{\pi(1)}$ along $\pi$ with respect to the pair $(f,g)$}. The transported matching will be denoted by the symbol
$T^{(f,g)}_\pi\left(\sigma_{(a,b)}\right)$.
More formally, we define
$T^{(f,g)}_\pi\left(\sigma_{\pi(0)}\right):\dgm\left(f_{\pi(1)}^*\right)\to \dgm\left(g_{\pi(1)}^*\right)$
by setting
$T^{(f,g)}_\pi\left(\sigma_{\pi(0)}\right):=T^g_\pi\circ\sigma_{\pi(0)}\circ \left(T^f_\pi\right)^{-1}$ (see Figure~\ref{dtm}).
We observe that $T^{(f,g)}_\pi$ is a map from
$\Sigma_{\pi(0)}^{(f,g)}$ to $\Sigma_{\pi(1)}^{(f,g)}$,
where $\Sigma_{(a, b)}^{(f,g)}$ is the set of all matchings from $\dgm\left(f_{(a,b)}^*\right)$ to $\dgm\left(g_{(a,b)}^*\right)$, with $(a,b)\in U$.

\begin{figure}
\[
\begin{tikzcd}
\dgm\left(f_{\pi(0)}^*\right) \arrow{rrrrrr}{\sigma_{\pi(0)}} \arrow{dd}{T^{f}_{\pi}} & & & & & & \dgm\left(g_{\pi(0)}^*\right) \arrow{dd}{T^{g}_{\pi}}\\
\\
\dgm\left(f_{\pi(1)}^*\right) \arrow{rrrrrr}{T^{(f,g)}_{\pi}\left(\sigma_{\pi(0)}\right) \,:=\, T^{g}_{\pi}\,\circ\, \sigma_{\pi(0)}\, \circ \,\left(T^{f}_{\pi}\right)^{-1}} & & & & & & \dgm\left(g_{\pi(1)}^*\right)
\end{tikzcd}
\]
\caption{The definition of $T^{(f,g)}_\pi\left(\sigma_{\pi(0)}\right)$.}
\label{dtm}
\end{figure}

\subsubsection{Some properties of the transport}
\label{spott}The next property trivially follows from the definition of transport.

\begin{proposition}\label{composition}
Let $\pi_1,\pi_2$ be two continuous paths in $U$, with $\pi_1(1)=\pi_2(0)$. Let $\pi_1*\pi_2$ be their composition, i.e. the continuous path $\pi_1\ast\pi_2:[0,1]\to U$ defined by setting $\pi_1\ast\pi_2(t):=\pi_1(2t)$ for $0\le t\le 1/2$ and $\pi_1\ast\pi_2(t):=\pi_2(2t-1)$ for $1/2\le t\le 1$. Then $T^f_{\pi_1\ast\pi_2}=T^f_{\pi_2}\circ T^f_{\pi_1}$.  As a consequence,
$T^{(f,g)}_{\pi_1\ast \pi_2}=
T^{(f,g)}_{\pi_2}\circ
T^{(f,g)}_{\pi_1}$.
\end{proposition}

In order to proceed, we need to recall the following result (Theorem 4.5 in \cite{CeDi*13}).

\begin{proposition}\label{oldbound}
Let $(a,b),(a',b')\in {\mathcal {P}}(\Lambda^+)$ with $|a-a'|,|b-b'|\le \eps<\min (a,1-a)$.
Then $d_B\left(\dgm \left(f^*_{(a,b)}\right),\dgm \left(f^*_{(a',b')}\right)\right)\leq\eps\cdot \frac{\|f\|_\infty + \|(a,1-a)\|_\infty + \|(b,-b)\|_\infty}{\min (a\cdot (a-\eps),(1-a)\cdot (1-a-\eps))}$.
\end{proposition}

Proposition~\ref{oldbound} allows to prove the following result, implying that the transport along a path in $U$ is continuous with respect to changes in the path.

\begin{proposition}\label{transcontinuity}
Let $f\in \mathcal{F}_{U,c}$.
Let $\bar \pi=\left(\bar a,\bar b\right):[0,1]\to U$ be a continuous path.
Choose a positive $\eta<\min_{t\in [0,1]}\min (\bar a(t),1-\bar a(t))$.
Let us define $$C_{\eta}:=\max_{t\in [0,1]}\frac{\|f\|_\infty + \|(\bar a(t),1-\bar a(t))\|_\infty + \|(\bar b(t),-\bar b(t))\|_\infty}{\min (\bar a(t)\cdot (\bar a(t)-\eta),(1-\bar a(t))\cdot (1-\bar a(t)-\eta))}.$$
If $\pi:[0,1]\to U$ is a continuous path such that $\pi(0)=\bar\pi(0)$ and $\left\|\bar\pi-\pi\right\|_\infty\le \min(\eta,c/ C_{\eta})$, then the inequality $\left\|T^f_{\bar\pi}(X)-T^f_{\pi}(X)\right\|_\infty\le \left\|\bar\pi-\pi\right\|_\infty
\cdot C_\eta$ holds for every $X\in \dgm\left(f_{\left(\bar a,\bar b\right)}^*\right)\setminus\{\Delta\}$.
\end{proposition}

\begin{proof}
By Proposition~\ref{uniquePath}, there is a unique path $\bar P:[0,1]\to\Delta^*$ induced by $\bar \pi$ such that $\bar P(0)=X$ and $\bar P(1)=T_{\bar\pi}^f(X)$, and a unique path $P:[0,1]\to\Delta^*$ induced by $\pi$ such that $P(0)=X$ and $P(1)=T_{\pi}^f(X)$.
Let us set
$\theta:= \max\left\{\tau\in [0,1]:\left\|T_{\bar\pi_\tau}^f(X)-T_{\pi_\tau}^f(X)\right\|_\infty\le
\|\bar\pi-\pi\|_{\infty}\cdot C_\eta \right\}$,
where the paths $\bar \pi_\tau,\pi_\tau:[0,1]\to U$ are defined by setting $\bar\pi_\tau(t):=\bar\pi(\tau t)$ and
$\pi_\tau(t):=\pi(\tau t)$ for every $t\in [0,1]$. In plain words, $T_{\bar\pi_\tau}^f(X)$ and $T_{\pi_\tau}^f(X)$ represent the transport of $X$ along $\bar\pi$ and $\pi$ with respect to $f$, respectively, for the time $\tau$ instead of the usual time $1$. We observe that $T_{\bar\pi_0}^f(X)=T_{\pi_0}^f(X)=X$.
Moreover, $\left\|\bar\pi_\tau-\pi_\tau\right\|_\infty\le\left\|\bar\pi-\pi\right\|_\infty$ for every $\tau\in [0,1]$.

If $\theta<1$, then on the one hand we can find a $\theta_+\in\ ]\theta,1]$ arbitrarily close to $\theta$ such that
$\left\|T_{\bar\pi_{\theta_+}}^f(X)-T_{\pi_{\theta_+}}^f(X)\right\|_\infty> \left\|\bar\pi-\pi\right\|_\infty
\cdot C_\eta$ and
$\left\|T_{\bar\pi_{\theta_+}}^f(X)-T_{\pi_{\theta_+}}^f(X)\right\|_\infty$ is arbitrarily close to $\left\|\bar\pi-\pi\right\|_\infty
\cdot C_\eta$.
We recall that $T_{\bar\pi_{\theta_+}}^f(X)\in \dgm\left(f^*_{\bar\pi_{\theta_+}(1)}\right)\setminus\{\Delta\}$ and
$T_{\pi_{\theta_+}}^f(X)\in \dgm\left(f^*_{\pi_{\theta_+}(1)}\right)\setminus\{\Delta\}$.

On the other hand, for every $t\in[0,1]$ the inequalities $$|\bar a(t)-a(t)|,|\bar b(t)-b(t)|\le \left\|\bar\pi-\pi\right\|_\infty\le \eta<\min (\bar a(t),1-\bar a(t))$$ hold.
For every positive $\delta\le \eta$, let us define $$C_{\delta}:=\max_{t\in [0,1]}\frac{\|f\|_\infty + \|(\bar a(t),1-\bar a(t))\|_\infty + \|(\bar b(t),-\bar b(t))\|_\infty}{\min (\bar a(t)\cdot (\bar a(t)-\delta),(1-\bar a(t))\cdot (1-\bar a(t)-\delta))}.$$
If we apply Proposition~\ref{oldbound} for $\eps:=\left\|\bar\pi-\pi\right\|_\infty$, we obtain the inequalities
$$d_B\left(\dgm\left(f_{\bar\pi_{ \theta_+}(1)}^*\right),\dgm\left(f_{\pi_{\theta_+}(1)}^*\right)\right)\leq \|\bar\pi-\pi\|_{\infty}\cdot C_{\left\|\bar\pi-\pi\right\|_\infty}\le \|\bar\pi-\pi\|_{\infty}\cdot C_\eta$$
so that a point $Y_{\theta_+}\in \dgm\left(f^*_{\pi_{\theta_+}(1)}\right)$ exists such that
$d\left(Y_{\theta_+},T_{\bar\pi_{\theta_+}}^f(X)\right)\le\|\bar\pi-\pi\|_{\infty}\cdot C_\eta\le c$, because
$\left\|\bar\pi-\pi\right\|_\infty\le c/
C_\eta$.
Since $f\in \mathcal{F}_{U,c}$, we have that $d\left(T_{\bar\pi_{\theta_+}}^f(X),\Delta\right)> 2c$. Therefore,
$d\left(Y_{\theta_+},\Delta\right)\ge d\left(T_{\bar\pi_{\theta_+}}^f(X),\Delta\right)-d\left(Y_{\theta_+},T_{\bar\pi_{\theta_+}}^f(X)\right)>2c-c=c$.
It follows that $Y_{\theta_+}\neq\Delta$ and
$$\left\|Y_{\theta_+}-T_{\bar\pi_{\theta_+}}^f(X)\right\|_\infty=d\left(Y_{\theta_+},T_{\bar\pi_{\theta_+}}^f(X)\right)
\le \|\bar\pi-\pi\|_{\infty}\cdot C_\eta
<
\left\|T_{\pi_{\theta_+}}^f(X)-T_{\bar\pi_{\theta_+}}^f(X)\right\|_\infty,$$ so that $Y_{\theta_+}\neq T_{\pi_{\theta_+}}^f(X)$.
Since $\theta_+$ is arbitrarily close to $\theta$, the 1D Stability Theorem~\ref{stability} implies that
a point $Z\in \dgm\left(f^*_{\pi_{\theta}(1)}\right)$ exists such that
the inequality $d\left(Z,T_{\bar\pi_{\theta}}^f(X)\right)\le \|\bar\pi-\pi\|_{\infty}\cdot C_\eta$ holds,
where $Z$ is the limit of the previously considered points $Y_{\theta_+}$. Furthermore,
$\left\|T_{\pi_{\theta}}^f(X)-T_{\bar\pi_{\theta}}^f(X)\right\|_\infty\le \|\bar\pi-\pi\|_{\infty}\cdot C_\eta$.
If $Z\neq T_{\pi_{\theta}}^f(X)$, then $\dgm\left(f^*_{\pi_{\theta}(1)}\right)$ contains at least two points ($Z$ and $T_{\pi_{\theta}}^f(X)$) that have a distance less than or equal to $\|\bar\pi-\pi\|_{\infty}\cdot C_\eta$ from $T_{\bar\pi_{\theta}}^f(X)$, and hence these two points have a distance less than or equal to $2\cdot \|\bar\pi-\pi\|_{\infty}\cdot C_\eta\le 2c$ from each other. If $Z= T_{\pi_{\theta}}^f(X)$, then this point is double in $\dgm\left(f^*_{\pi_{\theta}(1)}\right)$,
because of the construction of $Z$ and the inequality $Y_{\theta_+}\neq T_{\pi_{\theta_+}}^f(X)$.
Both cases contradict the assumption that $f\in \mathcal{F}_{U,c}$.

Therefore, if $\left\|\bar\pi-\pi\right\|_\infty\le \min(\eta,c/C_\eta)$ then $\theta=1$, and hence
we have that $\left\|T_{\bar\pi}^f(X)-T_{\pi}^f(X)\right\|_\infty=
\left\|T_{\bar\pi_1}^f(X)-T_{\pi_1}^f(X)\right\|_\infty
\le \left\|\bar\pi-\pi\right\|_\infty
\cdot C_\eta$.
\end{proof}

\subsection{Each loop in U induces a permutation on $\dgm\left(f_{(a,b)}^*\right)$}\label{loopspermut}

From the fact that the transport
along a path in $U$ is continuous with respect to changes in the path (consequence of Proposition~\ref{transcontinuity})
and the fact that $\dgm\left(f_{(a,b)}^*\right)\setminus \{\Delta\}$ is a discrete set,
the next result immediately follows.

\begin{proposition}\label{homotopy}
If two paths $\pi,\pi'$ in $U$ are homotopic to each other relatively to their common extrema, then $T^f_{\pi}\equiv T^f_{\pi'}$.
\end{proposition}

\begin{corollary}\label{loops}
The map $T^f$ taking each equivalence class $[\pi]$ to the permutation $T^f_{\pi}$ is a well-defined homomorphism from the fundamental group of $U$ at $(a,b)\in U$ to the group of permutations of $\dgm\left(f_{(a,b)}^*\right)$.
\end{corollary}

\begin{proof}
This follows from Propositions~\ref{composition} and~\ref{homotopy}.
\end{proof}

\begin{definition}
The image of the group homomorphism $T^f$ will be called the \emph{persistent monodromy group} of the filtering function $f$ with respect to $U$.
\end{definition}

In the introduction of this paper we recalled a basic example of filtering function $f:X=\R^2\to\R^2$ associated with a nontrivial persistent monodromy group. We observe that it is easy to adapt that example and obtain a normal filtering function $\hat f:M\to\R^2$ still associated with a nontrivial persistent monodromy group with respect to a suitable open set $U$, where $M$ is a smooth closed manifold.

\begin{proposition}\label{generators}
If the set $\{[\pi_j]\}_{j\in J}$ of homotopy classes of loops based at a point $(a,b)\in U$ is a set of generators for the fundamental group of $U$ at $(a,b)$, then the persistent monodromy group of $f$ with respect to $U$ is generated by the permutations $T^f_{\pi_j}$.
\end{proposition}

\begin{proof}
This follows from Corollary~\ref{loops}.
\end{proof}

The following interesting property holds.

\begin{proposition}\label{squareloop}
Let $\bar\pi:[0,1]\to U$ be a loop turning once around exactly one singular pair $\left(\bar a,\bar b\right)$ for $f$. Then $T^f_{\bar\pi}$ is either a transposition or the identity.
\end{proposition}

\begin{proof}
Since the image of $\bar\pi$ cannot contain singular pairs for $f$, Statement~(2) in Corollary~\ref{corsp}
implies that $\dgm\left(f^*_{\left(\bar a,\bar b\right)}\right)$ does not contain improper multiple points. From Statement~(1) in Corollary~\ref{corsp} it follows that $\left(\bar a,\bar b\right)$ is an isolated singular pair, and hence $B\left(\left(\bar a,\bar b\right),r\right)\cap \mathrm{Sing}(f)=\{\left(\bar a,\bar b\right)\}$ for every sufficiently small $r$, where $B(P,r)$ is the open ball of center $P$ and radius $r$ with respect to the Euclidean distance.
Let $V$ be the connected component of $\mathrm{Reg}(f)$ containing $U$.
Statement~(2) in Corollary~\ref{corsp} guarantees that the boundary of $V$ in ${\mathcal {P}}(\Lambda^+)$ is the union of a finite set and a (possibly empty) finite union of segments, whose points are singular pairs for $f$.
For every $r>0$, let us set $U_r:=V\setminus \bigcup_{P\in \mathrm{Sing}(f)} \overline{B(P,r)}$. Let us consider two values $r',r''$ with $0<r'< r''$, and the sets $U_{r'/2}$, $U_{r'}$, $U_{r''}$.
By definition of the set $\mathcal{F}_{U,c}$, we have that $\partial U$ does not contain any point of $\mathrm{Sing}(f)$. Therefore, we can choose $r',r''$ so small that $U\subseteq U_{r''}\subseteq U_{r'}\subseteq U_{r'/2}$ and the open sets $U_{r'/2},U_{r'},U_{r''}$ are connected.
Let us choose a $c'>0$ such that $f\in \mathcal{F}_{U_{r'/2},c'}$.
With reference to Proposition~\ref{transcontinuity}, let us choose a positive $\eta<\min_{t\in [0,1]}\min (a(t),1-a(t))$
for every path $\pi=(a,b):[0,1]\to B\left(\left(\bar a,\bar b\right),r''\right)\cap U_{r'/2}$. This value does not depend on $r'$.
Now, let us take a continuous path $\alpha:[0,1]\to U_{r'/2}$ such that $\alpha(0)=\bar\pi(0)$ and $\alpha(1)\in B\left(\left(\bar a,\bar b\right),r'\right)$. We also take a loop $\beta:[0,1]\to B\left(\left(\bar a,\bar b\right),r'\right)\cap U_{r'/2}$ such that $\beta(0)=\beta(1)=\alpha(1)$ and the loop $\gamma:=\alpha\ast\beta\ast\alpha^{-1}$ is homotopic to $\bar \pi$ in $U_{r'/2}$
(see Figure~\ref{alphaandbeta}). This loop $\beta$ exists because the open set bounded by the image of $\bar \pi$ does not contain singular pairs different from $\left(\bar a,\bar b\right)$.
Propositions~\ref{homotopy} and~\ref{composition} imply that $T^f_{\bar\pi}= T^f_{\gamma}
=\left({T^f_{\alpha}}\right)^{-1}\circ T^f_{\beta}\circ{T^f_{\alpha}}$.
Proposition~\ref{transcontinuity}
guarantees the existence of a constant $k$ such that if $r'$ is small enough, then for every loop $\beta':[0,1]\to B\left(\left(\bar a,\bar b\right),r'\right)\cap U_{r'/2}$
with $\beta'(0)=\beta(0)$
the value $\left\|T^f_{\beta}(X)-T^f_{\beta'}(X)\right\|_\infty$
is not greater than $kr'$, for every $X\in \dgm\left(f_{\left(\bar a,\bar b\right)}^*\right)\setminus\{\Delta\}$. We observe that $k$ does not depend on $r'$ and
Proposition~\ref{transcontinuity} does not require that the loops $\beta$ and $\beta'$ be homotopic in $U_{r'/2}$.
If we take $\beta'$
equal to the constant path having $\beta(0)$ as its image,
it follows that $\left\|T^f_{\beta}(X)-X\right\|_\infty=\left\|T^f_{\beta}(X)-T^f_{\beta'}(X)\right\|_\infty$ can be made arbitrarily small for every $X\in \dgm\left(f_{\beta(0)}^*\right)\setminus \{\Delta\}$, provided that $r'$ is small enough.
Since Corollary~\ref{cordpp} and the Stability Theorem~\ref{stability} imply that for $(a',b')\in B\left(\left(\bar a,\bar b\right),r'\right)\cap U_{r'/2}$ and $r'$ small enough the set $\dgm\left(f_{(a',b')}^*\right)\setminus \{\Delta\}$ contains only two proper points $P_{(a',b')}^1,P_{(a',b')}^2$ that are close to each other,
the fact that $T^f_{\beta}(X)$ must be close to $X$
guarantees that $T^f_{\beta}:\dgm\left(f_{\beta(0)}^*\right)\to \dgm\left(f_{\beta(0)}^*\right)$ is either the identity or the transposition exchanging $P_{\beta(0)}^1$ with $P_{\beta(0)}^2$.
Therefore, given that $T^f_{\alpha}:\dgm\left(f_{\alpha(0)}^*\right)\to \dgm\left(f_{\alpha(1)}^*\right)$ is a bijection, $T^f_{\bar\pi}=\left({T^f_{\alpha}}\right)^{-1}\circ T^f_{\beta}\circ{T^f_{\alpha}}$ is either the identity or a transposition.
\end{proof}

\begin{remark}\label{remgenerators}
Let $\left(\bar a,\bar b\right)$ be a regular point for $f$.
Let $\mathrm{Imp}(f)$ be the set of points $(a,b)\in{\mathcal {P}}(\Lambda^+)$ such that the persistence diagram $\dgm\left(f_{(a,b)}^*\right)$ contains at least one improper multiple point.
Let us consider the connected component $\widetilde{V}$ of $\left(\bar a,\bar b\right)$ in ${\mathcal {P}}(\Lambda^+)\setminus \mathrm{Imp}(f)$ and assume that $\widetilde{V}\cap \mathrm{Sing}(f)\neq \emptyset$. Corollary~\ref{corsp} implies that $\widetilde{V}\cap \mathrm{Sing}(f)$ is a finite set $\{(a_1,b_1),\ldots,(a_q,b_q)\}$. Let us take $r$ so small that the closed balls $\overline{B((a_j,b_j),2r)}$ are disjoint from each other and do not meet the boundary of $\widetilde{V}$ in ${\mathcal {P}}(\Lambda^+)$. We also require that $r$ is so small that $\left(\bar a,\bar b\right)\in U:=\widetilde{V}\setminus \bigcup_{P\in \mathrm{Sing}(f)} \overline{B(P,r)}$.
Then in $U$ for every singular pair $(a_j,b_j)$ we can find a loop $\pi_j$ based at $\left(\bar a,\bar b\right)$ that turns once around exactly $(a_j,b_j)$ and no other singular pair $(a_i,b_i)$.
The set of homotopy classes $\{[\pi_1],\ldots,[\pi_q]\}$ is a set of generators for the fundamental group of $U$ at $\left(\bar a,\bar b\right)$, so that
Propositions~\ref{generators} and \ref{squareloop} implicitly give a method to compute the persistent monodromy group of $f$ with respect to $U$. Furthermore, we know that if $G$ is a subgroup of the symmetric group $S_n$ and $G$ is generated by $m$ transpositions, then $|G|\le (m+1)!$. It follows that the cardinality of the image of $T^f$ is bounded by $(q+1)!$.
\end{remark}

\begin{figure}[ht]
\begin{tikzpicture}[scale=0.7]
\coordinate (A) at (-6,-5.5);
\coordinate (A1) at ($(A) + (10,-1)$);
\coordinate (A2) at ($(A1) + (3,8)$);
\coordinate (A3) at ($(A2) + (-7,6)$);
\coordinate (A4) at ($(A3) + (-8,-4)$);

\coordinate (B) at (-1.2,-1.7);
\coordinate (B1) at (1.3,-1.9);
\coordinate (B2) at (2.4,1);
\coordinate (B3) at (0.5,2.6);
\coordinate (B4) at (-1.4,2.0);
\coordinate (B5) at (-2.1,0);

\coordinate (AB) at ($(A)!0.5!(B)$);

\fill [gray!30] (-10,-8) rectangle (10,8);
\draw (0,0) circle (5cm);
\draw (0,0) circle (3cm);
\draw [black,fill=white] (0,0) circle (1.5cm);
\draw (0,0) node {$\bullet$};

\draw (0,5.35) node {{\large $\displaystyle{B\left(\left(\bar a,\bar b\right),r''\right)}$}};
\draw (0,3.35) node {{\large $\displaystyle{B\left(\left(\bar a,\bar b\right),r'\right)}$}};
\draw (0,0.50) node {{\large $\displaystyle{\left(\bar a,\bar b\right)}$}};
\draw (-8.5,7) node {{\Huge $U_{r'/2}$}};

\draw (A) node {$\bullet$};
\draw plot [smooth,tension=0.6] coordinates {(A) (A1) (A2) (A3) (A4) (A)};
\draw (A1) -- ($(A1) + (-0.2,0.07)$);
\draw (A1) -- ($(A1) + (-0.1,-0.2)$);
\draw ($(A1) + (-0.6,-0.5)$) node {{\large $\bar \pi$}};

\draw (B) node {$\bullet$};
\draw (A) .. controls ($(AB) + (0.7,-0.7)$) .. (B);

\coordinate (X) at ($(A) + (2.35,1)$);
\draw (X) -- ($(X) + (-0.2,0.05)$);
\draw (X) -- ($(X) + (-0.05,-0.2)$);
\draw ($(X) + (-0.65,-0.05)$) node {{\large $\alpha$}};

\draw plot [smooth,tension=0.9] coordinates {(B) (B1) (B2) (B3) (B4) (B5) (B)};
\draw (B4) -- ($(B4) + (0,0.25)$);
\draw (B4) -- ($(B4) + (0.25,0)$);
\draw ($(B4) + (0.8,0.3)$) node {{\large $\beta$}};
\end{tikzpicture}
\caption{The path $\alpha:[0,1]\to U_{r'/2}$ and the loop $\beta:[0,1]\to B\left(\left(\bar a,\bar b\right),r'\right)\cap U_{r'/2}$ used in the proof of Proposition~\ref{squareloop}.}
\label{alphaandbeta}
\end{figure}

Before proceeding we recall that the sets $U$ and $\mathcal{F}_{U,c}$ have been chosen, as described at the end of Section~\ref{preparing}.

The next result implies that the transport
along a path in $U$ is continuous with respect to changes in the filtering function.

\begin{proposition}\label{stability_of_transport}
Let $f,g\in \mathcal{F}_{U,c}$ with $\|f-g\|_\infty< c$. If $\pi:[0,1]\to U$ is a continuous path, and $X\in \dgm\left(f_{\pi(0)}^*\right)$, $Y\in \dgm\left(g_{\pi(0)}^*\right)$ are two points whose distance is less than or equal to $\|f-g\|_\infty$, then $\|T^g_\pi(Y)-T^f_\pi(X)\|\le \|f-g\|_\infty$.
\end{proposition}

\begin{proof}\label{proof_stability_of_transport}
We know that
$\left\|f_{(a,b)}^*- g_{(a,b)}^*\right\|_\infty\le \|f-g\|_\infty$ for every $(a,b)\in {\mathcal {P}}(\Lambda^+)$
(Lemma~\ref{lemmafab}).
For every $t\in [0,1]$, the Stability Theorem~\ref{stability} implies that
for each point $X_t\in \dgm\left(f_{\pi(t)}^*\right)$ there is a unique point $Y_t(X_t)\in \dgm\left(g_{\pi(t)}^*\right)$ having distance from $X_t$ less than or equal to $\|f-g\|_\infty< c$. We observe that if this point were not unique, two points of $\dgm\left(g_{\pi(t)}^*\right)$ would exist, with a distance from each other less than $2c$, against the definition of the set $\mathcal{F}_{U,c}$.
Obviously, $Y_0(X)=Y$.

For every $\tau\in [0,1]$, let $\pi_\tau:[0,1]\to U$ be the path defined by setting $\pi_\tau(t)=\pi(\tau t)$. Let us consider the set $S$ of the values
$\tau$ such that $T^g_{\pi_\tau}(Y_0(X))= Y_\tau(T^f_{\pi_\tau}(X))$.
Since $T^g_{\pi_0}$ and $T^f_{\pi_0}$ are identity maps, we observe that $T^g_{\pi_0}(Y_0(X))=Y_0(X)=Y_0(T^f_{\pi_0}(X))$, so that $0\in S$.
We can consider the number $s:=\max S$.
If $s <1$, we can find an arbitrarily small $\delta>0$ such that the inequality $T^g_{\pi_{s +\delta}}(Y_0(X))\neq Y_{s +\delta}(T^f_{\pi_{s +\delta}}(X))$ holds. Since
$\|T^g_{\pi_{s }}(Y_0(X))-T^f_{\pi_{s }}(X)\|=\|Y_s (T^f_{\pi_{s }}(X))-T^f_{\pi_{s }}(X)\|\le \|f-g\|_\infty< c$,
 we have that
$\|T^g_{\pi_{s +\delta}}(Y_0(X))-T^f_{\pi_{s +\delta}}(X)\|< c$, provided that $\delta$ is small enough. Moreover, we have that $\|Y_{s +\delta}(T^f_{\pi_{s +\delta}}(X))-T^f_{\pi_{s +\delta}}(X)\|\le \|f-g\|_\infty< c$. It follows that
$\|Y_{s +\delta}(T^f_{\pi_{s +\delta}}(X))-T^g_{\pi_{s +\delta}}(Y_0(X))\|_\infty< 2c$ with
 $T^g_{\pi_{s +\delta}}(Y_0(X))\neq Y_{s +\delta}(T^f_{\pi_{s +\delta}}(X))$ and
$T^g_{\pi_{s +\delta}}(Y_0(X)),Y_{s +\delta}(T^f_{\pi_{s +\delta}}(X))\in \dgm\left(g_{\pi_{s +\delta}(1)}^*\right)$. This fact contradicts our assumption that
$g\in \mathcal{F}_{U,c}$. Therefore $s =1$, so that
$\|T^g_{\pi}(Y)-T^f_{\pi}(X)\|=\|T^g_{\pi}(Y_0(X))-T^f_{\pi}(X)\|=\|Y_{1}(T^f_{\pi}(X))-T^f_{\pi}(X)\|\le \|f-g\|_\infty$.
\end{proof}

We conclude this section by observing that the transport operator $T^f_{\pi}$ cannot exchange the positions of improper points, under the assumptions $f\in \mathcal{F}_{U,c}$ and $\pi:[0,1]\to U$. This is due to the fact that $T^f_{\pi}$ moves the points at infinity along a line, so that in order to exchange their positions those points should collide. As a consequence, the path $\pi$ should meet the set $\mathrm{Sing}(f)$, against our assumptions. In plain words, we could say that the phenomenon of monodromy concerns only the proper points of persistence diagrams.


\subsection{The definition of the coherent matching distance}
\label{secdefCMD}

\begin{definition}
\label{cohcost}
Let $\Pi_{(a,b)}(U)$ be the set of all continuous paths $\pi:[0,1]\to U$ with $\pi(0)=(a,b)$.
If $\sigma_{(a, b)}\in\Sigma_{(a,b)}^{(f,g)}$, the \emph{coherent cost} of $\sigma_{(a, b)}$ is the value
$${\mathrm{cohcost}_U}\left(\sigma_{(a,b)}\right):=\sup_{\pi\in \Pi_{(a,b)}(U)}\mathrm{cost}\left(T^{(f,g)}_{\pi}\left(\sigma_{(a,b)}\right)\right).$$
\end{definition}

The following proposition states that the function ${\mathrm{cohcost}_U}$ is invariant under transport.

\begin{proposition}
\label{Cohinvariance}
Let $\sigma_{(a, b)}:\dgm\left(f_{(a,b)}^*\right)\to\dgm\left(g_{(a,b)}^*\right)$ be a matching, with $(a,b)\in U$.
If $\pi':[0,1]\to U$ is a continuous path with $\pi'(0)=(a,b)$, then
${\mathrm{cohcost}_U}\left(T^{(f,g)}_{\pi'}\left(\sigma_{(a,b)}\right)\right)={\mathrm{cohcost}_U}\left(\sigma_{(a,b)}\right)$.
\end{proposition}

\begin{proof}
By recalling Proposition~\ref{composition} we have
\begin{align*}
{\mathrm{cohcost}_U}\left(T^{(f,g)}_{\pi'}\left(\sigma_{(a,b)}\right)\right) &=
\sup_{\pi\in \Pi_{\pi'(1)}(U)}\mathrm{cost}
\left(
T^{(f,g)}_{\pi}\left(T^{(f,g)}_{\pi'}\left(\sigma_{(a,b)}\right)\right)
\right)\\
&= \sup_{\pi\in \Pi_{\pi'(1)}(U)}\mathrm{cost}
\left(
T^{(f,g)}_{\pi'\ast\pi}\left(\sigma_{(a,b)}\right)
\right)\\
&= \sup_{\pi\in \Pi_{(a,b)}(U)}\mathrm{cost}
\left(
T^{(f,g)}_{\pi}\left(\sigma_{(a,b)}\right)
\right)\\
&= {\mathrm{cohcost}_U}\left(\sigma_{(a,b)}\right).
\end{align*}
\end{proof}

The set $\Sigma_{(a, b)}^{(f,g)}$ is finite because of Corollary~\ref{corfiniteness}. Therefore, we can give the following definition.

\begin{definition}
\label{CD}
Let $(a,b)\in U$.
The \emph{coherent 2-dimensional matching distance} between $f$ and $g$ is defined as
\[
CD_U(f,g) = \min_{\sigma_{(a, b)}\in\Sigma_{(a, b)}^{(f,g)}} {\mathrm{cohcost}_U}\left(\sigma_{(a,b)}\right).
\]
\end{definition}

\begin{proposition}\label{prop_CDU_dnd}
$CD_U(f,g)$ does not depend on the basepoint $(a,b)$.
\end{proposition}
\begin{proof}\label{proof_prop_CDU_dnd}
This immediately follows from Proposition~\ref{Cohinvariance}.
\end{proof}

\begin{proposition}\label{prop_CDU_pd}
$CD_U(f,g)$ is a pseudo-distance.
\end{proposition}
\begin{proof}\label{proof_prop_CDU_pd}
For every $(a,b)\in U$, let ${\mathrm{id}}_{(a,b)}\in \Sigma_{(a, b)}^{(f,f)}$ be the identity matching. The property $CD_U(f,f)=0$ follows from the fact that $T^{(f,f)}_{\pi}\left({\mathrm{id}}_{\pi(0)}\right)={\mathrm{id}}_{\pi(1)}$ for every
continuous path $\pi:[0,1]\to U$, implying the equality ${\mathrm{cohcost}_U}\left({\mathrm{id}}_{(a,b)}\right)=0$.
To show symmetry, we observe that
$T^{(f,g)}_{\pi}\left(\sigma_{\pi(0)}\right)\circ
T^{(g,f)}_{\pi}\left(\sigma_{\pi(0)}^{-1}\right)=
T^g_\pi\circ\sigma_{\pi(0)}\circ \left(T^f_\pi\right)^{-1}\circ
T^f_\pi\circ\left(\sigma_{\pi(0)}\right)^{-1}\circ \left(T^g_\pi\right)^{-1}={\mathrm{id}}_{\pi(1)}$ for every
continuous path $\pi:[0,1]\to U$.
It follows that
$T^{(g,f)}_{\pi}\left(\sigma_{\pi(0)}^{-1}\right)=
\left(T^{(f,g)}_{\pi}\left(\sigma_{\pi(0)}\right)\right)^{-1}$.
By recalling that $\mathrm{cost}\left(\sigma_{(a,b)}^{-1}\right)=\mathrm{cost}\left(\sigma_{(a,b)}\right)$ for every $(a,b)\in U$ and every $\sigma_{(a, b)}\in\Sigma_{(a, b)}^{(f,g)}$, we have that
${\mathrm{cohcost}_U}\left(\sigma_{(a,b)}^{-1}\right)=
{\mathrm{cohcost}_U}\left(\sigma_{(a,b)}\right)$, and so $CD_U(f,g) = CD_U(g,f)$.
As for the triangle inequality, let $f,g,h\in {\mathcal{F}}_{U,c}$. If $\sigma_{(a, b)}\in\Sigma_{(a, b)}^{(f,g)}$ and $\tau_{(a, b)}\in\Sigma_{(a, b)}^{(g,h)}$, then
\begin{align*}
{\mathrm{cohcost}_U}\left(\tau_{(a,b)}\circ\sigma_{(a,b)}\right) &= \sup_{\pi\in \Pi_{(a,b)}(U)}\mathrm{cost}\left(T^{(f,h)}_{\pi}\left(\tau_{(a,b)}\circ\sigma_{(a,b)}\right)\right)\\
&= \sup_{\pi\in \Pi_{(a,b)}(U)}\mathrm{cost}\left(T^h_\pi\circ\tau_{(a,b)}\circ\sigma_{(a,b)}\circ \left(T^f_\pi\right)^{-1}\right)\\
&= \sup_{\pi\in \Pi_{(a,b)}(U)}\mathrm{cost}\left(T^h_\pi\circ\tau_{(a,b)}
\circ \left(T^g_\pi\right)^{-1}\circ T^g_\pi\circ\sigma_{(a,b)}\circ \left(T^f_\pi\right)^{-1}\right)\\
&\le \sup_{\pi\in \Pi_{(a,b)}(U)}\left(\mathrm{cost}\left(T^h_\pi\circ\tau_{(a,b)}
\circ \left(T^g_\pi\right)^{-1}\right)\right.\\
&\qquad \left.+
\mathrm{cost}\left(T^g_\pi\circ\sigma_{(a,b)}\circ \left(T^f_\pi\right)^{-1}\right)\right)\\
&\le \sup_{\pi\in \Pi_{(a,b)}(U)}\mathrm{cost}\left(T^h_\pi\circ\tau_{(a,b)}
\circ \left(T^g_\pi\right)^{-1}\right)\\
&\qquad +
\sup_{\pi\in \Pi_{(a,b)}(U)}\mathrm{cost}\left(T^g_\pi\circ\sigma_{(a,b)}\circ \left(T^f_\pi\right)^{-1}\right)\\
&= {\mathrm{cohcost}_U}\left(\tau_{(a,b)}\right)+
{\mathrm{cohcost}_U}\left(\sigma_{(a,b)}\right).
\end{align*}
Therefore, if $\sigma_{(a, b)}\in\Sigma_{(a, b)}^{(f,g)}$ and $\tau_{(a, b)}\in\Sigma_{(a, b)}^{(g,h)}$ the inequality
$${\mathrm{cohcost}_U}\left(\tau_{(a,b)}\circ\sigma_{(a,b)}\right)\le
{\mathrm{cohcost}_U}\left(\tau_{(a,b)}\right)+
{\mathrm{cohcost}_U}\left(\sigma_{(a,b)}\right)$$ holds.
Since $\tau_{(a,b)}\circ\sigma_{(a,b)}\in\Sigma_{(a, b)}^{(f,h)}$, it follows that
$CD_U(f,h)\le
{\mathrm{cohcost}_U}\left(\tau_{(a,b)}\right)+
{\mathrm{cohcost}_U}\left(\sigma_{(a,b)}\right)$ for every $\sigma_{(a, b)}\in\Sigma_{(a, b)}^{(f,g)}$ and for every $\tau_{(a, b)}\in\Sigma_{(a, b)}^{(g,h)}$. Hence
$CD_U(f,h)\le CD_U(g,h)+CD_U(f,g)$.
\end{proof}

The next result shows that the coherent 2-dimensional matching distance is stable, in a suitable sense.

\begin{theorem}\label{thmstability}
If $f,g\in \mathcal{F}_{U,c}$ and $\|f-g\|_\infty< c$, then $CD_U(f,g)\le\|f-g\|_{\infty}$.
\end{theorem}
\begin{proof}\label{proofprop3}
Fix $\left(\bar a,\bar b\right)\in U$ and take the matching $\sigma_{\left(\bar a,\bar b\right)}:\dgm\left(f_{\left(\bar a,\bar b\right)}^*\right)\to \dgm\left(g_{\left(\bar a,\bar b\right)}^*\right)$ obtained by
taking each point $X\in \dgm\left(f_{\left(\bar a,\bar b\right)}^*\right)$ to the unique point in $\dgm\left(g_{(a,b)}^*\right)$ having distance from $X$ less than or equal to $\|f-g\|_\infty< c$. Since
$\mathrm{cost}\left(\sigma_{\left(\bar a,\bar b\right)}\right)\le \|f-g\|_\infty$, Proposition~\ref{stability_of_transport} implies that
${\mathrm{cohcost}_U}\left(\sigma_{\left(\bar a,\bar b\right)}\right)\le \|f-g\|_\infty$.
\end{proof}

\begin{remark}\label{remcov}
The definition of our coherent matching distance could be easily expressed by means of the concept of universal covering $\mathcal{C}$ of $U$. In fact, each homotopy class of paths based at $\left(\bar a,\bar b\right)\in U$ that is relative to their endpoints corresponds to a point in $\mathcal{C}$. If $U$ is replaced by $\mathcal{C}$ in our construction, we have that any matching defined at a point of $\mathcal{C}$ can be transported in a unique way to any other point of $\mathcal{C}$. The replacement of the parameter set $U$ with its universal covering would naturally lead to an equivalent definition of coherent matching distance. In the present exposition, we preferred to maintain the set $U$ for the sake of simplicity.
\end{remark}

\subsubsection{Computational aspects}\label{compasp}
Let us consider the open set $U$ and the finite set $\{(a_1,b_1),\ldots,(a_q,b_q)\}$ of singular pairs for $f$ defined in Remark~\ref{remgenerators}, with reference to a basepoint $\left(\bar a,\bar b\right)\in U$. Let us choose another point $(a,b)\in U$ and fix a continuous path $\pi_{(a,b)}:[0,1]\to U$ with $\pi_{(a,b)}(0)=\left(\bar a,\bar b\right)$ and $\pi_{(a,b)}(1)=(a,b)$.
Let $\Pi_{\left(\bar a,\bar b\right)\leadsto (a,b)}(U)$ and $L_{\left(\bar a,\bar b\right)}(U)$ be the sets of all continuous paths in $U$ from $\left(\bar a,\bar b\right)$ to $(a,b)$ and of loops $\pi:[0,1]\to U$ with $\pi(0)=\left(\bar a,\bar b\right)$, respectively.
We observe that each  $\pi\in \Pi_{\left(\bar a,\bar b\right)\leadsto (a,b)}(U)$
is homotopic relatively to its extrema $\left(\bar a,\bar b\right)$, $(a,b)$ to the path $\pi'\ast \pi_{(a,b)}$,
where $\pi':=\pi\ast \pi_{(a,b)}^{-1}\in L_{\left(\bar a,\bar b\right)}(U)$.
From Propositions~\ref{composition} and~\ref{homotopy} it follows that
\begin{align*}
\sup_{\pi\in \Pi_{\left(\bar a,\bar b\right)\leadsto (a,b)}(U)}\mathrm{cost}\left(T^{(f,g)}_{\pi}\left(\sigma_{(a,b)}\right)\right) &= \sup_{\pi'\in L_{\left(\bar a,\bar b\right)}(U)}\mathrm{cost}\left(T^{(f,g)}_{\pi'\ast \pi_{(a,b)}}\left(\sigma_{(a,b)}\right)\right)\\
&= \sup_{\pi'\in L_{\left(\bar a,\bar b\right)}(U)}\mathrm{cost}
\left(
T^{(f,g)}_{\pi_{(a,b)}}\left(T^{(f,g)}_{\pi'}\left(\sigma_{(a,b)}\right)\right)
\right).
\end{align*}
For every index $j\in \{1,\ldots,q\}$, we can choose
a loop $\pi_j:[0,1]\to U$ starting at the regular point $\left(\bar a,\bar b\right)$ and turning once around $(a_j,b_j)$ but not around other singular pairs for $f$. Then the
set of matchings $\left\{T^{(f,g)}_{\pi'}\left(\sigma_{\left(\bar a,\bar b\right)}\right):\pi'\in L_{\left(\bar a,\bar b\right)}(U)\right\}$
equals the set of all matchings that can be written as
$T^g_{\pi'}\circ\sigma_{\left(\bar a,\bar b\right)}\circ \left(T^f_{\pi'}\right)^{-1}$ with $\pi'\in L_{\left(\bar a,\bar b\right)}(U)$, i.e. as
$$T^g_{\pi_{j_r}}\circ\cdots \circ T^g_{\pi_{j_1}}\circ\sigma_{\left(\bar a,\bar b\right)}\circ
\left(T^f_{\pi_{j_1}}\right)^{-1}\circ\cdots \circ \left(T^f_{\pi_{j_r}}\right)^{-1}$$ with $j_1,\ldots,j_r\in \{1,\ldots,q\}$, because of Proposition~\ref{generators}.
It follows that the computation of $\sup_{\pi\in \Pi_{\left(\bar a,\bar b\right)\leadsto (a,b)}(U)}\mathrm{cost}\left(T^{(f,g)}_{\pi}\left(\sigma_{(a,b)}\right)\right)$ just requires to manage the finite set of paths $\{\pi_1,\ldots,\pi_q,\pi_{(a,b)}\}$.

\section{A maximum principle for the coherent transport}\label{proofconjecture}

We are now ready to prove the most important property of the coherent transport.
Let us consider the set $\overline{\Pi}$ of all paths in $U$ starting at a fixed pair $\left(\bar a,\bar b\right)$ and ending at a variable pair with abscissa different from $\bar a$. In this section we show that the value $\mathrm{cost}\left(T_{\pi}^{(f,g)}\left(\sigma_{\left(\bar a,\bar b\right)}\right)\right)$ involved in the definitions of ${\mathrm{cohcost}_U}$ and $CD_U$ satisfies a sort of maximum principle as a function in the variable $\pi\in\overline{\Pi}\cup\{\bar\pi\}$, where $\bar\pi$ is the constant path at $\left(\bar a,\bar b\right)$. Indeed, we are going to prove that if $\bar\pi$ is a point of strict local maximum for the function
$\mathrm{cost}\left(T_{\pi}^{(f,g)}\left(\sigma_{\left(\bar a,\bar b\right)}\right)\right)$ varying
$\pi$ in $\overline{\Pi}\cup\{\bar\pi\}$ (up to homotopies of $\pi$ relative to its endpoints), then $\bar a$ must equal $\frac{1}{2}$ (Theorem~\ref{maxprinc}). As a consequence of this statement, we will prove the main result of this section (Theorem~\ref{finalth}), casting new light on the question presented at the beginning of this paper.
Before proceeding, we recall that in this paper the symbol $\Sigma_{(a, b)}^{(f,g)}$ denotes the set of all matchings from $\dgm\left(f_{(a,b)}^*\right)$ to $\dgm\left(g_{(a,b)}^*\right)$, with $(a,b)\in U$ and $f,g\in \mathcal{F}_{U,c}$,
while $\Pi_{\left(\bar a,\bar b\right)\leadsto (a,b)}(U')$ is the set of all continuous paths in $U'$ from $\left(\bar a,\bar b\right)$ to $(a,b)$, for an open set $U'$ of ${\mathcal {P}}(\Lambda^+)$ containing both $\left(\bar a,\bar b\right)$ and $(a,b)$.

\begin{theorem}\label{a<>1/2}
 Let $f,g\in \mathcal{F}_{U,c}$ and $\left(\bar a,\bar b\right)\in U$. If $\sigma_{\left(\bar a,\bar b\right)}\in\Sigma_{\left(\bar a,\bar b\right)}^{(f,g)}$ with $\mathrm{cost}\left(\sigma_{\left(\bar a,\bar b\right)}\right) < \infty$ and $V$ is an open subset of $U$ containing $\left(\bar a,\bar b\right)$, then these two properties hold:
\begin{enumerate}
  \item If $\bar a<\frac{1}{2}$ then there exist a point $(a',b')\in V$ with $\bar a<a'<\frac{1}{2}$ and a path $\pi'\in \Pi_{\left(\bar a,\bar b\right)\leadsto (a',b')}(V)$ such that
$
\mathrm{cost}\left(T_{\pi'}^{(f,g)}\left(\sigma_{\left(\bar a,\bar b\right)}\right)\right)\ge \mathrm{cost}\left(\sigma_{\left(\bar a,\bar b\right)}\right)
$.
  \item If $\bar a>\frac{1}{2}$ then there exist a point $(a',b')\in V$ with $\frac{1}{2}<a'<\bar a$ and a path $\pi'\in \Pi_{\left(\bar a,\bar b\right)\leadsto (a',b')}(V)$ such that
$
\mathrm{cost}\left(T_{\pi'}^{(f,g)}\left(\sigma_{(\bar a,\bar b)}\right)\right)\ge \mathrm{cost}\left(\sigma_{\left(\bar a,\bar b\right)}\right)
$.
\end{enumerate}
\end{theorem}

\begin{proof}
Let us prove $(1)$, since the proof of $(2)$ is completely analogous.

The value $\mathrm{cost}\left(\sigma_{\left(\bar a,\bar b\right)}\right)$ is given by the distance $d(A,B)$ between a point $A\in\dgm\left(f_{\left(\bar a,\bar b\right)}^*\right)$ and a point $B\in\dgm\left(g_{\left(\bar a,\bar b\right)}^*\right)$.
By possibly exchanging the roles of $A$ and $B$, we can assume $A$ not closer than $B$ to $\Delta$.

We first treat the case $A, B\neq\Delta$, so that we can write $A=(u_A,v_A)$ and $B=(u_B,v_B)$, with $u_A<v_A$, $u_B<v_B$, $u_A<\infty$, $u_B<\infty$ and $v_A,v_B\le \infty$. In particular, note that $\mathrm{cost}\left(\sigma_{\left(\bar a,\bar b\right)}\right) < \infty$ implies that either $v_A,v_B < \infty$ or $v_A,v_B = \infty$. If $v_A,v_B < \infty$, by the Position Theorem~\ref{GT*}, and recalling that $\bar a<\frac{1}{2}$, we know that four points $(x_A,y_A),(\xi_A,\eta_A)\in r_{\left(\bar a,\bar b\right)}\cap\Gamma(f)$ and $(x_B,y_B),(\xi_B,\eta_B)\in r_{\left(\bar a,\bar b\right)}\cap\Gamma(g)$ exist, for which
\begin{eqnarray}\label{pippo}
u_A = x_A - \bar b,\ v_A  = \xi_A - \bar b,\ u_B = x_B - \bar b,\ v_B = \xi_B - \bar b.
\end{eqnarray}
Since $A$ is not closer than $B$ to $\Delta$, $v_A-u_A  \geq v_B-u_B$ (i.e. $\xi_A - x_A\geq \xi_B - x_B$).
Moreover, it is not restrictive to assume that $|u_A-u_B|\ge |v_A-v_B|$ (i.e. $|x_A-x_B|\ge |\xi_A-\xi_B|$), as the proof works analogously if  $|v_A-v_B|\ge |u_A-u_B|$ (i.e. $|\xi_A-\xi_B|\ge |x_A-x_B|$).
We also observe that $x_A<\xi_A$ and $x_B<\xi_B$, since $u_A<v_A$ and $u_B<v_B$.

We can find an open ball $W$ entirely contained in $V$ and centered at $\left(\bar a,\bar b\right)$.
Let us take a pair $(a',b')\in W$ such that $(x_A,y_A) \in r_{(a',b')}$ and $\bar a<a'<\frac{1}{2}$. Let
$\pi':[0,1]\to W$ be the straight path from $\left(\bar a,\bar b\right)$ to $(a',b')$, parameterized by $s\in [0,1]$ with equation
\begin{equation}\label{path}
\pi'(s) = (1-s) \cdot \left(\bar a,\bar b\right) + s\cdot (a',b').
\end{equation}
Following the path $\pi'$ defined in equation~(\ref{path}), the admissible line $r_{(a',b')}$ is thus obtained by rotating $r_{\left(\bar a,\bar b\right)}$ around the point $(x_A,y_A)$, in a way that the slope of $r_{\left(\bar a,\bar b\right)}$ progressively decreases while approaching $r_{(a',b')}$.
(Analogously, in the case $|v_A-v_B|\ge |u_A-u_B|$ we should take a pair $(a',b')\in W$ such that $(\xi_A,\eta_A) \in r_{(a',b')}$ and $\bar a<a'<\frac{1}{2}$, and rotate the line $r_{\left(\bar a,\bar b\right)}$ around the point $(\xi_A,\eta_A)$ in a way that the slope of $r_{\left(\bar a,\bar b\right)}$ progressively decreases while approaching $r_{(a',b')}$.)
By definition of transported matching (Section~\ref{deftm}), the matching $T^{(f,g)}_{\pi'}\left(\sigma_{\left(\bar a,\bar b\right)}\right): \dgm\left(f^*_{(a',b')}\right)\to\dgm\left(g^*_{(a',b')}\right)$ must match $A':=T_{\pi'}^f(A)$ to $B':=T_{\pi'}^g(B)$. Obviously, $\pi'\in \Pi_{\left(\bar a,\bar b\right)\leadsto (a',b')}(V)$ and $\mathrm{cost}\left(T^{(f,g)}_{\pi'}\left(\sigma_{\left(\bar a,\bar b\right)}\right)\right)\geq d(A',B')$.
Let us denote $A'$ and $B'$ by $(u_{A'},v_{A'})$ and $(u_{B'},v_{B'})$, respectively.

We need to show that $d(A',B')\geq d(A,B)$, so proving that the inequality
$\mathrm{cost}\left(T_{\pi'}^{(f,g)}\left(\sigma_{\left(\bar a,\bar b\right)}\right)\right)
\ge \mathrm{cost}\left(\sigma_{\left(\bar a,\bar b\right)}\right)
$
holds in case $A, B\neq\Delta$ and $v_A,v_B < \infty$. To do this, we only consider the case $d(A,B)>0$, as the case $d(A,B) = 0$ is trivial.

Recall now the transport of $A$ and $B$ induced by the same path $\pi'$. The Position Theorem~\ref{GT*} implies that four points $(x_{A'},y_{A'})$, $(\xi_{A'},\eta_{A'})\in r_{(a',b')}\cap\Gamma(f)$ and $(x_{B'},y_{B'})$, $(\xi_{B'},\eta_{B'})\in r_{(a',b')}\cap\Gamma(g)$ exist such that
$$
u_{A'} = x_{A'} - b',\ v_{A'}  = \xi_{A'} - b',\ u_{B'} = x_{B'} - b',\ v_{B'} = \xi_{B'}- b'.
$$
Note that $x_{A'} = x_A$, by construction of $\pi'$. Now, if $x_A < x_B$ it necessarily follows that $x_{B'} \ge x_{B}$: Indeed, this is a direct consequence of Proposition~\ref{uniquePath}, the Position Theorem~\ref{GT*} and the structure of $\Gamma(g)$ (see Figure~\ref{figmainproof}). In particular, $x_{B} = x_{B'}$ if and only if both $(x_{B},y_{B})$ and $(x_{B'},y_{B'})$ belong to the same vertical, improper contour of $g$. Analogously, if $x_A > x_B$ it follows that $x_{B'} \leq x_{B}$. Therefore, in all cases we have $|x_A-x_B| \leq |x_{A'}-x_{B'}|$, so that
{\setlength\arraycolsep{2pt}
$$
\begin{array}{ccccc}
\max\left\{|x_A - x_B|, |\xi_A-\xi_B|\right\} &   =   & |x_A-x_B| & &\vspace*{1mm}\\
& \leq & |x_{A'}-x_{B'}| & \leq & \max\left\{|x_{A'} - x_{B'}|, |\xi_{A'}-\xi_{B'}|\right\}.
\end{array}
$$
A similar reasoning holds for the relation between $x_A=x_{A'}, \xi_A$ and $\xi_{A'}$. Precisely, from $x_A < \xi_A$ it necessarily follows that $\xi_{A'} \geq \xi_{A}$. Thus $\xi_{A}-x_A\leq\xi_{A'}-x_{A'}$.
By recalling that $\xi_A - x_A\ge \xi_B - x_B$, we can write}
{\setlength\arraycolsep{2pt}
$$
\begin{array}{ccccc}
\max\left\{\dfrac{\xi_A - x_A}{2}, \dfrac{\xi_B - x_B}{2}\right\} &   =   & \dfrac{\xi_A - x_A}{2} & &\vspace*{1mm}\\
& \leq & \dfrac{\xi_{A'} - x_{A'}}{2} & \leq & \max\left\{\dfrac{\xi_{A'} - x_{A'}}{2}, \dfrac{\xi_{B'} - x_{B'}}{2}\right\}.
\end{array}
$$
The definition of $d$ (cf.~(\ref{deltaDistance})) and our assumptions state that
\begin{align}
d(A,B)&= d\left(\left(u_A,v_A\right),\left(u_B,v_B)\right)\right) \\
      &:=\min\left\{\max\left\{\left|u_A-u_B\right|,\left|v_A-v_B\right|\right\},\max\left\{\frac{v_A-u_A}{2},\frac{v_B-u_B}{2}\right\}\right\} \nonumber\\
      &=\min\left\{\max\left\{\left|x_A-x_B\right|,\left|\xi_A-\xi_B\right|\right\},\max\left\{\frac{\xi_A-x_A}{2},\frac{\xi_B-x_B}{2}\right\}\right\}, \nonumber
\label{eq:1}
\end{align}
\begin{align}
d(A',B')&= d\left(\left(u_{A'},v_{A'}\right),\left(u_{B'},v_{B'})\right)\right) \\
      &:=\min\left\{\max\left\{\left|u_{A'}-u_{B'}\right|,\left|v_{A'}-v_{B'}\right|\right\},\max\left\{\frac{v_{A'}-u_{A'}}{2},\frac{v_{B'}-u_{B'}}{2}\right\}\right\} \nonumber\\
      &=\min\left\{\max\left\{\left|x_{A'}-x_{B'}\right|,\left|\xi_{A'}-\xi_{B'}\right|\right\},\max\left\{\frac{\xi_{A'}-x_{A'}}{2},\frac{\xi_{B'}-x_{B'}}{2}\right\}\right\}. \nonumber
\label{eq:2}
\end{align}
Therefore, $d(A',B') \geq d(A,B)$ for $A, B\neq\Delta$ and $v_A,v_B < \infty$.

The case when $A, B\neq\Delta$ and $v_A=v_B= \infty$ can be treated analogously, after setting $\xi_A=\xi_{A'}=\xi_B=\xi_{B'}=\infty$
and observing that $d(A,B)=|u_A-u_B|$. Also in this case we get $d(A',B') \geq d(A,B)$.

Suppose now that $A\neq\Delta$ and $B=\Delta$, so that $d(A,B)=d(A,\Delta)$ and $B'=T_{\pi'}^g(B) = \Delta$ by choosing $(a',b')$ and $\pi'$ as we did above. It is easy to see that
$d(A',\Delta) =\frac{\xi_{A'} - x_{A'}}{2}\geq \frac{\xi_A - x_A}{2}=d(A,\Delta)$ with $x_{A}=x_{A'}$, i.e. $d(A',B') \geq d(A,B)$.

If $A=B=\Delta$ then $A=A'=B=B'=\Delta$, so that $d(A',B') = d(A,B)$.

(In these steps, the meaning of the symbols is the same established in the previous part of the proof.)

In all cases $d(A',B') \ge d(A,B)$, so that
$\mathrm{cost}\left(T_{\pi'}^{(f,g)}\left(\sigma_{\left(\bar a,\bar b\right)}\right)\right)
\ge
\mathrm{cost}\left(\sigma_{\left(\bar a,\bar b\right)}\right)$, since
$\mathrm{cost}\left(T_{\pi'}^{(f,g)}\left(\sigma_{\left(\bar a,\bar b\right)}\right)\right)\ge d(A',B')$ and
$d(A,B) =\mathrm{cost}\left(\sigma_{\left(\bar a,\bar b\right)}\right)$.
}
\end{proof}

\begin{figure}
\begin{center}
\begin{tikzpicture}[scale=1.0]
\draw [red,line width=1.2] (4.25,-1.55) arc (40:58:11);
\draw [red,line width=1.2] (4.25,0.46) arc (30:70:5);
\draw [->,line width=0.8] (2.73,2.8) arc (70:58:3);

\draw [red,line width=1] (1.65,1) -- (6,1);
\draw [red,line width=1] (3.5,-1.55) -- (3.5,6.45);

\draw [blue,line width=1.2] (2,0.5) -- (3.7,6.45);
\draw [blue,line width=1.2] (2,0.5) -- (5.9,6.02);
\draw (2.38,3.7) node {$r_{\left(\bar a,\bar b\right)}$};
\draw (4.7,3.2) node {$r_{(a',b')}$};

\draw [line width=0.1,black,fill=white] (2,0.5) circle (0.08);
\draw (2.8,0.55) node {$(x_A,y_A)$};

\draw [line width=0.1,green,fill=green] (2.13,1) circle (0.08);
\draw [line width=0.1,green,fill=green] (2.35,1) circle (0.08);

\draw [line width=0.1,black,fill=black] (2.5,2.25) circle (0.08);
\draw [line width=0.1,black,fill=black] (3,1.9) circle (0.08);

\draw [line width=0.1,yellow,fill=yellow] (3.5,2.65) circle (0.08);
\draw [line width=0.1,yellow,fill=yellow] (3.5,5.8) circle (0.08);

\end{tikzpicture}
\end{center}\caption{A passage in the proof of Theorem~\ref{a<>1/2} (case $\bar a<\frac{1}{2}$, $v_A,v_B<\infty$, $|x_A-x_B|\ge |\xi_A-\xi_B|$, $x_A< x_B$, with $A,B\neq \Delta$ and $A$ not closer than $B$ to $\Delta$): While rotating $r_{(a,b)}$ toward $r_{(a',b')}$, the abscissa of the intersection of $r_{(a,b)}$ with $\Gamma(g)$ cannot decrease, i.e. $x_{B'}\ge x_{B}$. Here the extended Pareto grid $\Gamma(g)$ is represented in red. The green, black and yellow pairs of points refer to the possible locations of $(x_B,y_B)$ and $(x_{B'},y_{B'})$.}
\label{figmainproof}
\end{figure}
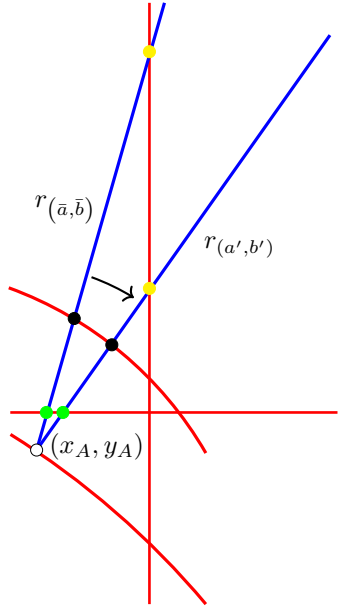

As a direct consequence of Theorem~\ref{a<>1/2}, we have the following result.

\begin{theorem}[Maximum Principle]\label{maxprinc}
Let $f,g\in \mathcal{F}_{U,c}$, $\left(\bar a,\bar b\right)\in U$ and
$\sigma_{\left(\bar a,\bar b\right)}\in\Sigma_{\left(\bar a,\bar b\right)}^{(f,g)}$ with $\mathrm{cost}\left(\sigma_{\left(\bar a,\bar b\right)}\right) < \infty$. If an open set $W\subseteq U$ exists, such that $\left(\bar a,\bar b\right)\in W$ and for all $(a,b)\in W$ with $a\neq \bar a$ the inequality
$
\mathrm{cost}\left(\sigma_{\left(\bar a,\bar b\right)}\right) > \mathrm{cost}\left(T_{\pi}^{(f,g)}\left(\sigma_{\left(\bar a,\bar b\right)}\right)\right)
$
holds for every path $\pi\in \Pi_{\left(\bar a,\bar b\right)\leadsto (a,b)}(W)$, then $\bar a=\frac{1}{2}$.
\end{theorem}

\begin{remark}\label{remstraightlines}
If we consider a convex open set $V\subseteq U$ and assume $\left(\bar a,\bar b\right),(a',b')\in V$, then any
path $\pi\in \Pi_{\left(\bar a,\bar b\right)\leadsto (a',b')}(V)$ and the straight path $\pi'(s) = (1-s) \cdot \left(\bar a,\bar b\right) + s\cdot (a',b')$ for $s\in[0,1]$ are homotopic to each other relatively to their common extrema. Therefore, $T^f_{\pi}\equiv T^f_{\pi'}$
(Proposition~\ref{homotopy}). As a consequence, the statements of Theorem~\ref{a<>1/2} and Theorem~\ref{maxprinc} could be reformulated in terms of coherent transport along straight lines. We preferred not to do that for the sake of simplicity in our exposition.
\end{remark}

Now we can answer the main question presented at the beginning of this paper, provided that $D_{\mathrm{match}}$ is replaced with $CD_U$.

\begin{theorem}\label{finalth}
Let $f,g\in \mathcal{F}_{U,c}$ and $\left(\bar a,\bar b\right)\in U$.
Assume that the closure $\overline{U}$ of $U$ in $\mathbb{R}^2$ is a compact set contained in the open set ${\mathcal {P}}(\Lambda^+)=\,]0,1[\,\times\R$, and that $\partial U$ is a $C^1$-submanifold of $\mathbb{R}^2$.
Then a matching $\sigma_{\left(\bar a,\bar b\right)}\in\Sigma_{\left(\bar a,\bar b\right)}^{(f,g)}$, a point $\left(\hat a,\hat b\right)\in \overline{U}$ and a continuous path $\hat \pi:[0,1]\to\overline{U}$ from $\left(\bar a,\bar b\right)$ to $\left(\hat a,\hat b\right)$ exist such that
\begin{enumerate}
\item $\hat\pi([0,1[)\subseteq U$;
\item $\mathrm{cost}\left(T_{\hat \pi}^{(f,g)}\left(\sigma_{\left(\bar a,\bar b\right)}\right)\right)=
\mathrm{cohcost}_U\left(\sigma_{\left(\bar a,\bar b\right)}\right)=
CD_U(f,g)$;
\item $\left(\hat a,\hat b\right)\in \partial U$ or $\hat a =\frac{1}{2}$.
\end{enumerate}
\end{theorem}

\begin{proof}
Let $\sigma_{\left(\bar a,\bar b\right)}\in\Sigma_{\left(\bar a,\bar b\right)}^{(f,g)}$ with
$CD_U(f,g) = {\mathrm{cohcost}_U}\left(\sigma_{\left(\bar a,\bar b\right)}\right)$.
For every $i\in \N$, let us consider a pair $(a_i,b_i)\in U$ and a continuous path $\pi_i\in \Pi_{\left(\bar a,\bar b\right)\leadsto (a_i,b_i)}(U)$, such that
$\lim_{i\to \infty}
\mathrm{cost}\left(T_{\pi_i}^{(f,g)}\left(\sigma_{\left(\bar a,\bar b\right)}\right)\right)={\mathrm{cohcost}_U}\left(\sigma_{\left(\bar a,\bar b\right)}\right)$.
Since $\overline{U}$ is compact, we can find a subsequence of $\left((a_i,b_i)\right)$ converging to a point $\left(a^\sharp,b^\sharp\right)\in\overline{U}$.
Because of the assumption that $\partial U$ is a $C^1$-submanifold of $\mathbb{R}^2$, for every index $i$ we can find
a continuous path $\pi'_i:[0,1]\to \overline{U}$ from $(a_i,b_i)$ to $\left(a^\sharp,b^\sharp\right)$ with $\pi'_i\left([0,1[\right)\subseteq U$, such that the maximum distance between the points in the set $\pi'_i\left([0,1]\right)$ and the point $\left(a^\sharp,b^\sharp\right)$ tends to $0$ when $i$ tends to infinity.
Let us now take another open set $U'\subseteq {\mathcal {P}}(\Lambda^+)$ and a $c'>0$ such that $f,g\in \mathcal{F}_{U',c'}$ and $\overline{U}\subseteq U'\subseteq\overline{U'}\subseteq{\mathcal {P}}(\Lambda^+)$, where $\overline{U'}$ is the closure of $U'$ in $\mathbb{R}^2$.
The uniform continuity of the transport along a path in $U'$ with respect to changes in the path (Proposition~\ref{transcontinuity}) implies that
$\lim_{i\to \infty} \left|\mathrm{cost}\left(T_{\pi_i\ast\pi'_i}^{(f,g)}\left(\sigma_{\left(\bar a,\bar b\right)}\right)\right)-
\mathrm{cost}\left(T_{\pi_i}^{(f,g)}\left(\sigma_{\left(\bar a,\bar b\right)}\right)\right)\right|=0$.
Hence, the equality $\lim_{i\to \infty} \mathrm{cost}\left(T_{\pi_i\ast\pi'_i}^{(f,g)}\left(\sigma_{\left(\bar a,\bar b\right)}\right)\right)=
{\mathrm{cohcost}_U}\left(\sigma_{\left(\bar a,\bar b\right)}\right)$ holds.
Since the set $\Sigma_{\left(a^\sharp,b^\sharp\right)}^{(f,g)}$ is finite,
by possibly extracting a subsequence we can assume that all the matchings
$T_{\pi_i\ast\pi'_i}^{(f,g)}\left(\sigma_{\left(\bar a,\bar b\right)}\right)$ coincide.
It follows that for every index $i$ the path $\pi_i\ast\pi'_i$ is a continuous path in $\overline{U}$ from $\left(\bar a,\bar b\right)$ to $\left(a^\sharp,b^\sharp\right)$ with $\pi_i\ast\pi'_i([0,1[)\subseteq U$, such that
$\mathrm{cost}\left(T_{\pi_i\ast\pi'_i}^{(f,g)}\left(\sigma_{\left(\bar a,\bar b\right)}\right)\right)={\mathrm{cohcost}_U}\left(\sigma_{\left(\bar a,\bar b\right)}\right)$.

Let us now consider the set $Z$ of all pairs $(a,b)\in \overline{U}$ for which
a continuous path $\pi:[0,1]\to\overline{U}$ from $\left(\bar a,\bar b\right)$ to $(a,b)$ exists, such that $\pi\left([0,1[\right)\subseteq U$ and $\mathrm{cost}\left(T_{\pi}^{(f,g)}\left(\sigma_{\left(\bar a,\bar b\right)}\right)\right)={\mathrm{cohcost}_U}\left(\sigma_{\left(\bar a,\bar b\right)}\right)$.
We have just seen that $\left(a^\sharp,b^\sharp\right)\in Z$, so that $Z\neq\emptyset$.
Once again because of the assumption that $\partial U$ is a $C^1$-submanifold of $\mathbb{R}^2$ and Proposition~\ref{transcontinuity}, $Z$ is compact.
This can be proved by means of the same reasoning we have previously seen. Indeed, let us take a sequence $\left(\left(a^\sharp_i,b^\sharp_i\right)\right)$ in $Z$ and a sequence $\left(\pi^\sharp_i\right)$ of continuous paths $\pi^\sharp_i:[0,1]\to\overline{U}$ from $\left(\bar a,\bar b\right)$ to $\left(a^\sharp_i,b^\sharp_i\right)$, such that $\pi^\sharp_i\left([0,1[\right)\subseteq U$ and $\mathrm{cost}\left(T_{\pi^\sharp_i}^{(f,g)}\left(\sigma_{\left(\bar a,\bar b\right)}\right)\right)={\mathrm{cohcost}_U}\left(\sigma_{\left(\bar a,\bar b\right)}\right)$.
Since $\overline{U}$ is compact, we can assume that the sequence $\left(\left(a^\sharp_i,b^\sharp_i\right)\right)$ converges to a point $\left(a^\star,b^\star\right)\in \overline{U}$.
Starting from the previous sequences, by applying Proposition~\ref{transcontinuity} to the set $U'$ and recalling
the assumption that $\partial U$ is a $C^1$-submanifold of $\mathbb{R}^2$
it is easy to construct a sequence $\left(\pi^\flat_i\right)$ of continuous paths $\pi^\flat_i:[0,1]\to\overline{U}$ from $\left(\bar a,\bar b\right)$ to $\left(a^\star,b^\star\right)$, such that $\pi^\flat_i\left([0,1[\right)\subseteq U$, and $\lim_{i\to\infty}\mathrm{cost}\left(T_{\pi^\flat_i}^{(f,g)}\left(\sigma_{\left(\bar a,\bar b\right)}\right)\right)={\mathrm{cohcost}_U}\left(\sigma_{\left(\bar a,\bar b\right)}\right)$.
Since the set $\Sigma_{\left(a^\star,b^\star\right)}^{(f,g)}$ is finite,
by possibly extracting a subsequence from $\left(\pi^\flat_i\right)$ we can assume that all the matchings
$T_{\pi^\flat_i}^{(f,g)}\left(\sigma_{\left(\bar a,\bar b\right)}\right)$ coincide.
It follows that
$\mathrm{cost}\left(T_{\pi^\flat_i}^{(f,g)}\left(\sigma_{\left(\bar a,\bar b\right)}\right)\right)={\mathrm{cohcost}_U}\left(\sigma_{\left(\bar a,\bar b\right)}\right)$ for every index $i$.
Therefore, $\left(a^\star,b^\star\right)\in Z$ and hence $Z$ is compact.

Let us now take a pair $\left(\hat a,\hat b\right)\in Z$ that minimizes the distance from the line $a=\frac{1}{2}$,
and a continuous path $\hat \pi$ in $\overline{U}$ from $\left(\bar a,\bar b\right)$ to $\left(\hat a,\hat b\right)$
such that $\hat \pi([0,1[)\subseteq U$ and $\mathrm{cost}\left(T_{\hat\pi}^{(f,g)}\left(\sigma_{\left(\bar a,\bar b\right)}\right)\right)={\mathrm{cohcost}_U}\left(\sigma_{\left(\bar a,\bar b\right)}\right)$.
If $\left(\hat a, \hat b\right)\in\partial U$, our proof is concluded.
If $\left(\hat a,\hat b\right)\in U$ and $\hat a \neq \frac{1}{2}$, Theorem~\ref{a<>1/2} implies
the existence of a pair $(a',b')\in U$ with $\left|a'-\frac{1}{2}\right|<\left|\hat a-\frac{1}{2}\right|$ and a continuous path $\pi'\in \Pi_{\left(\hat a,\hat b\right)\leadsto (a',b')}(U)$
such that
${\mathrm{cohcost}_U}\left(\sigma_{\left(\bar a,\bar b\right)}\right)=\mathrm{cost}\left(T_{\hat\pi}^{(f,g)}\left(\sigma_{\left(\bar a,\bar b\right)}\right)\right) \le \mathrm{cost}\left(T_{\pi'}^{(f,g)}\left(T_{\hat \pi}^{(f,g)}\left(\sigma_{\left(\bar a,\bar b\right)}\right)\right)\right)$.
Since $\hat \pi\ast\pi'\in \Pi_{\left(\bar a,\bar b\right)\leadsto (a',b')}(U)$,
$\mathrm{cost}\left(T_{\pi'}^{(f,g)}\left(T_{\hat\pi}^{(f,g)}\left(\sigma_{\left(\bar a,\bar b\right)}\right)\right)\right)\le{\mathrm{cohcost}_U}\left(\sigma_{\left(\bar a,\bar b\right)}\right)$ because of the definition of $\mathrm{cohcost}_U$, and hence
$\mathrm{cost}\left(T_{\pi'}^{(f,g)}\left(T_{\hat\pi}^{(f,g)}\left(\sigma_{\left(\bar a,\bar b\right)}\right)\right)\right)=
{\mathrm{cohcost}_U}\left(\sigma_{\left(\bar a,\bar b\right)}\right)$. The\-re\-fore $(a',b')\in Z$ and the distance of
$(a',b')$ from the line $a=\frac{1}{2}$ is strictly less than the distance of
$\left(\hat a,\hat b\right)$ from the same line, against the choice of $\left(\hat a,\hat b\right)$. Hence, if $\left(\hat a,\hat b\right)\in U$ then the equality $\hat a =\frac{1}{2}$ must hold.
\end{proof}

\begin{remark}\label{finalremark}
After proving the maximum principle for the coherent transport, one could think that the use of the 2-dimensional parameter space ${\mathcal {P}}(\Lambda^+)$ is useless, and that the study of 2D persistence diagrams should be introduced by means of a 1-dimensional parameter space from the very beginning, only taking account of lines of slope $1$. This opinion is not correct, because the use of a 1-dimensional parameter space would hide the phenomenon of monodromy, since in a 1-dimensional space it would not be possible to turn around a singular point.
\end{remark}

\section{Relation between the coherent matching distance and the classical matching distance}
\label{relCUDmatch}

In this section we want to explore some relations between the coherent matching distance and the classical matching distance.
We start by stating the following simple result. 

\begin{proposition}
\label{simpleprop}
Let $c>0$.
Let $\mathcal{U}:=\{U_i\}$ be a finite family of open and connected subsets of ${\mathcal {P}}(\Lambda^+)$. Then the function $CD_{\mathcal{U}}:=\max _iCD_{U_i}$ is a pseudo-metric on the set $\bigcap_i \mathcal{F}_{U_i,c}$.
\end{proposition}

\begin{proof}
It is a trivial consequence of the fact that the maximum of pseudo-distances is a pseudo-distance.
\end{proof}

For every finite family $\mathcal{U}:=\{U_i\}$ of disjoint open and connected subsets of ${\mathcal {P}}(\Lambda^+)$, the classical counterpart of the pseudo-metric $CD_{\mathcal{U}}$
is the pseudo-metric ${D_{\mathrm{match},}}_\mathcal{U}$ defined by setting
$$
{D_{\mathrm{match},}}_\mathcal{U}(f,g)=\sup_{(a,b)\in \bigcup_i U_i}d_B\left(\dgm\left(f_{(a,b)}^*\right),\dgm\left(g_{(a,b)}^*\right)\right)
.$$

The following statement holds.

\begin{proposition}
\label{relationCDDmatch}
Let $c>0$.
Let $\mathcal{U}:=\{U_i\}$ be a finite family of disjoint open and connected subsets of ${\mathcal {P}}(\Lambda^+)$.
If $f,g\in \bigcap_i \mathcal{F}_{U_i,c}$, then ${D_{\mathrm{match},}}_\mathcal{U}(f,g)\le CD_{\mathcal{U}}(f,g)$.
\end{proposition}

\begin{proof}
Let us assume by contradiction that ${D_{\mathrm{match},}}_\mathcal{U}(f,g)> CD_{\mathcal{U}}(f,g)$. By definition of ${D_{\mathrm{match},}}_\mathcal{U}(f,g)$ we can find a real number $\epsilon>0$, an index $j$ and a point $(a,b)\in U_j$ such that $\mathrm{cost}(\sigma) \geq CD_{U_j}(f,g)+\epsilon$ for any matching $\sigma$ between $\dgm\left(f_{(a,b)}^*\right)$ and $\dgm\left(g_{(a,b)}^*\right)$. On the other hand, by Definition~\ref{cohcost} we have that ${\mathrm{cohcost}_{U_j}}\left(\sigma\right) \geq \mathrm{cost}(\sigma)$ for any such $\sigma$, implying that  $CD_{U_j}(f,g) \geq CD_{U_j}(f,g)+\epsilon$, thus getting a contradiction.
\end{proof}

\section{Conclusions}
In this paper we have presented a new theoretical framework for metric comparison in 2D persistent homology. In particular, we have illustrated the concept of coherent matching distance and studied some of its properties. In order to do that, we have also introduced the concepts of extended Pareto grid and transport of a matching, and we have shown their use to manage the phenomenon of monodromy. Finally, we have proved some theorems that make clear the importance of filtrations associated with lines of slope $1$ in 2D persistent homology.

In our opinion, many problems should deserve further research. First of all, it would be interesting to extend the presented concepts to filtering functions taking values in $\R^m$ with $m>2$. Secondly, the genericity of our assumptions concerning the extended Pareto grid should possibly be proved. Thirdly, the relation between the classical multidimensional matching distance $D_{\mathrm{match}}$ and the coherent matching distance $CD_U$ could be investigated further. Finally, methods for the efficient computation of the coherent matching distance should be developed.

We plan to devote further papers to these topics.

\section*{Acknowledgement}
Work carried out under the auspices of INdAM-GNSAGA. M.E. has been partially supported by the Toposys project FP7-ICT-318493-STREP, as well as an ESF Short Visit grant under the Applied and Computational Algebraic Topology networking programme. A.C. is partially supported by the FP7 Integrated Project IQmulus, FP7-ICT-2011–318787, and the H2020 Project Gravitate, H2020 - REFLECTIVE - 7 - 2014 - 665155. The authors are grateful to Claudia Landi for her valuable advice. This paper is dedicated to the memory of Ola Mouaffek Shihab Eddin, Naya Raslan and Abish Masih.

\bibliographystyle{amsplain}

\bibliography{amp_CD_bib}

\end{document}